\documentclass[reqno,11pt]{amsart}

\usepackage[marginparwidth=5em]{geometry}
\usepackage{amssymb}
\usepackage{xcolor}
\usepackage{amsthm}
\usepackage{mathrsfs}
\usepackage{amsmath}
\usepackage{enumerate}
\usepackage{amsfonts}
\usepackage{soul}
\usepackage{hyperref}
\usepackage{comment}
\usepackage{mathtools}
\usepackage[hyperpageref]{backref}
\newcommand{\R}{\mathbb{R}}

\newtheorem{thm}{Theorem}[section]
\newtheorem{lem}[thm]{Lemma}
\newtheorem{prp}[thm]{Proposition}
\newtheorem{cor}[thm]{Corollary}
\newtheorem{observation}[thm]{Remark}
\newtheorem{definition}[thm]{Definition}
\DeclareMathOperator{\Div}{div}
\numberwithin{equation}{section}
\theoremstyle{definition}
\newtheorem*{ack}{Acknowledgments}
\begin{document}


\title[On a fractional semilinear Neumann problem arising in Chemotaxis]{On a fractional semilinear Neumann problem arising in Chemotaxis}

\author[E. Cinti]{Eleonora Cinti}
\address{Dipartimento di Matematica\newline\indent
	Alma Mater Studiorum Universit\`a di Bologna
	\newline\indent
	piazza di Porta San Donato, 5\newline\indent
	40126 Bologna - Italy}
\email{eleonora.cinti5@unibo.it} 

\author[M. Talluri]{Matteo Talluri}
\address{Dipartimento di Matematica\newline\indent
	Alma Mater Studiorum Universit\`a di Bologna
	\newline\indent
	piazza di Porta San Donato, 5\newline\indent
	40126 Bologna - Italy}
\email{matteo.talluri@unibo.it} 

\begin{abstract}
	We study a semilinear and nonlocal Neumann problem, which is the fractional analogue of the problem considered by Lin--Ni--Takagi in the '80s. The model under consideration arises in the description of stationary configurations of the Keller--Segel model for chemotaxis, when a nonlocal diffusion for the concentration of the chemical is considered. In particular, we extend to any fractional power $s\in (0,1)$ of the Laplacian (with homogeneous Neumann boundary conditions) the results obtained in \cite{stinga2015fractional} for $s=1/2$. We prove existence and some qualitative properties of non--constant solutions when the diffusion parameter $\varepsilon$ is small enough, and on the other hand, we show that for $\varepsilon$ large enough any solution must be necessarily constant. 
	
\end{abstract}

\maketitle

\section{Introduction}
In this paper we consider the following fractional semilinear Neumann problem on a smooth and bounded domain $\Omega\subset\mathbb{R}^n$,
\begin{equation}\label{main system}\begin{cases}
(-\varepsilon\Delta_N)^{s}u+u=u^p\quad&\text{on $\Omega,$}\\
\partial_\nu u=0\quad&\text{in $\partial \Omega,$}\\
u>0&\text{in $\Omega$},
\end{cases}\end{equation}
where $\varepsilon>0,$ $s\in (0,1)$, $1<p<\frac{n+2s}{n-2s}$,  $\partial_\nu$ denote the outward normal derivative to $\partial \Omega$ and $(-\Delta_N)^s$ denotes the Spectral fractional Laplacian.

This problem can be seen as a possible nonlocal analogue of the one studied by Lin, Ni, and Takagi in a series of seminal papers in the eighties (\cite{lin1988large, NT}). The case $s=1/2$ was considered by Stinga and Volzone in \cite{stinga2015fractional}, here we extend their results to any fractional power $s\in (0,1)$ of the Neumann Laplacian.

There are different possible generalization of the classical local problem to the non local setting, depending on which notion of \textit{fractional Laplacian} and of \textit{Neumann boundary condition}  one considers.
As already mention, here we focuse on the \textit{spectral Neumann Laplacian}, as done in \cite{stinga2015fractional}, since this is interesting in applications.
{At the end of this Introduction, we will comment more in detail on other possible notions, describing their differences and similarities.}

Thanks to the results in \cite{stinga2010fractional,stinga2010extension}, in order to study problem \eqref{main system}, we will study the associated local extended problem:

\begin{equation}\label{main-intro}
	\begin{cases}\varepsilon \Delta_x U+\frac{1-2s}{y}U_y+U_{y y}=0 & \text { in } \mathcal{C}, \\ \partial_\nu U=0 & \text { on } \partial_L \mathcal{C}, \\
    U(x,0)>0& \text { in } \Omega\\
     -\lim_{y\to 0}y^{1-2s}U_y(x, y)+U(x, 0)=U(x, 0)^p & \text { in } \Omega,\end{cases}
\end{equation}
where $\mathcal C:=\Omega \times (0,+\infty)$ and $\partial_L\mathcal C:=\partial \Omega \times (0,+\infty)$.
Notice that here, just for convenience, we have preferred to put the parameter $\varepsilon$ as a coefficient in the equation, rather than on the weighted normal derivative on $\{y=0\}$, see Remark \ref{epsilon}.

Let
$$a:=1-2s.$$
Problem \eqref{main system} has a variational structure and the associated 
 energy functional is given by  \[\begin{split}
	J_{\varepsilon,a}(U):=\frac{1}{2}\iint_{\mathcal{C}}\left(\varepsilon^{}|\nabla_xU|^2+|U_y|^2\right)y^a dxdy+\int_{\Omega}|U(x,0)|^2dx -\frac{1}{p+1}\int_{\Omega}|U(x,0)|^{p+1}dx
\end{split}.\]

Our main results concern exitence and non--existence of non-constant solutions for problem \eqref{main-intro}, depending on the parameter $\varepsilon$.

More precisely, we have the following:


\begin{thm}\label{existence result intro}
	There exists $\overline \varepsilon>0$ such that if  $\varepsilon<\overline \varepsilon$, then  problem (\ref{main system}) admits at least one non--constant positive weak solution $u^{\varepsilon,a}$.  Moreover,
	if we denote by $U^{\varepsilon,a}$ its extension satisfying \eqref{main-intro}, then we have the following energy estimate:
	 \begin{equation}\label{energy}J_{\varepsilon,a}(U^{\varepsilon,a})\le C(n,a,p,\Omega)\varepsilon^{\frac{n}{2}}.\end{equation}
\end{thm}

For the notion of weak solution, see Section 4.

In Section 6, we also prove some properties of such solutions, in particular we show that they tend to zero in measure and concentrates around some points as $\varepsilon\rightarrow 0$.

On the other hand, we show non--existence of non--constant solutions for sufficiently large $\varepsilon$, according to the following result

\begin{thm}\label{non-existence-intro}
	There exists $\varepsilon^*>0$ such that if $\varepsilon>\varepsilon^*$ then the unique weak solution of problem (\ref{main system}) is $u\equiv1.$
\end{thm}

To prove existence of a weak solution to \eqref{main-intro} (and thus of \eqref{main system}), we make a standard use of the Mountain Pass Theorem, then for showing that such solution cannot be constant if $\varepsilon$ is small enough, the energy estimate \eqref{energy} (proven by means of a suitable competitor) will be crucial.

On the other hand, the proof of Theorem \ref{non-existence-intro} crucially relies on some a priori (uniform in $\varepsilon$) $L^\infty$-estimates for weak solutions, which will be obtained by a blow-up technique \`a la Gidas-Spruck (see Section 5).
More precisely, we have the following

\begin{thm}\label{uniformly bounded-intro}
	  There exists a positive constant $C>0$, that does not depend on $\varepsilon$, such that, if $u$ is a weak solution of \eqref{main system},
	then,
      \[\sup_{x\in\Omega}u(x)\le C.\] 
\end{thm}

{With respect to the case $s=1/2$ studied in \cite{stinga2015fractional}, we emphasize that here the extended function $U$ is not harmonic but satisfies a weighted equation that, for $\varepsilon=1$, reads as:
$$\text{div}(y^{1-2s}\nabla U)=0.$$

Observe that the weight $y^{1-2s}$ could be singular or degenerate (depending on whether $s$ is above or below $1/2$), nevertheless it enjoys some good properties since it belongs to the Muckenhoupt class $A_2$ (see \cite[Page 79]{MR643158}). In particular,  this requires some attention when using some results that are classical tools when dealing with harmonic functions, for example Harnack inequality, Hopf's Lemma, regularity results via semigroup theory (see Section 3 for further comments on this last point).}

We describe now the Keller-Segel model and how problem \eqref{main system} arises in the study of Chemotaxis.
	Chemotaxis is the oriented movement of cells towards higher concentrations of some
	chemical substance in their environment.
Keller and Segel proposed the following model to describe the 
chemotactic aggregation.

Let  $w(x, t)>0$ be the density of a bacteria population and
$v(x, t)>0$ the density of a certain chemoattractant, then we have
\[\begin{cases}
	\partial_t w= D_1\Delta w-\chi \text{div}(w\nabla(\log v)) & \mbox{in } \Omega \times (0,+\infty)\\
	\partial_t v=D_2\Delta v - aw +b v  & \mbox{in } \Omega\times (0,+\infty)\\
	\partial_\nu w=\partial_\nu v=0  &\mbox{on } \partial\Omega.
	
	\end{cases}\]

    In \cite{lin1988large, NT}, Lin, Ni, and Takagi studied the steady states of the above system, relying on the observation that it could be reduced to the study of a single equation. More precisely,
we can rewrite the stationary equation
    $$
    D_1\Delta w-\chi \text{div}(w\nabla(\log v))=0
    $$
 as \begin{equation}\label{KS}\text{div}(D_1 w \nabla(\log w- \chi D_1^{-1}\log v))=0\end{equation} 
	 and use the $0$--Neumann condition, to see that $w =\lambda v^{{\chi/D_1}}$ for some positive $\lambda$.
     {Indeed, multiplying \eqref{KS} by $\log w- \chi D_1^{-1}\log v$, integrating by parts, and using that $\partial_v w=\partial_\nu v=0$ on $\partial \Omega$, one gets
     $$
     \int_\Omega D_1 w |\nabla (\log w- \chi D_1^{-1}\log v)|^2\,dx =0, 
     $$
     which implies that $\log w- \chi D_1^{-1}\log v$ must be constant.
     }
     
     Hence the stationary version of the Keller--Segel system reduces to a single equation for $v$, which can be written in the form
	 \[\begin{cases}
	 	-\varepsilon \Delta u + u = u^p &\text{in } \Omega\\
	 	\partial_\nu u=0 &\text{on } \partial\Omega,
	 	
	 \end{cases}\]
	where $p=\chi/D_1$ and $\varepsilon=D_2/a$.

Possibly different nonlocal approaches to the study of Chemotaxis have been proposed recently in several papers. 

In \cite{Escudero2006}, Escudero proposed a model in which the chemical satisfies a classical local diffusion equation while the bacteria satisfy a nonlocal diffusion. 
In contrast, Stinga and Volzone \cite{stinga2015fractional}, considered a model in which the diffusion is assumed to be local for the bacteria and nonlocal for the chemical. As already mentioned, in this work we use the same approach, generalizing their results to any fractional power $s\in (0,1)$. 
Another nonlocal model for chemotaxis was proposed by  Huaroto and Neves in \cite{Huaroto-Neves}. They considered a fractional variant of the Keller-Segel model, in which the velocity (that in the classical model is proportional to the gradient of the concentration of the chemoattractant $\nabla v$), is instead defined through a linear fractional operator. This leads again to  an equation for the chemical, in which the diffusion is governed by a fractional Neumann Laplacian.

{To conclude this Introduction, we recall other possible notions of nonlocal Neumann Laplacian, describing their differences and similarities, with a particular emphasis on their probabilistic interpretation (for more details, see e.g. \cite[Sections 2, 4]{AV}, \cite[Section 7]{DROV} and references therein).

We start by recalling the probabilistic interpreptation of the \textit{spectral Neumann Laplacian}, considered in this paper.  Given a bounded domain $\Omega \subset \mathbb R^n$ and a point $x\in \Omega$,
let
$u(x, t)$ denote the number of particles located at the point $x$ at time $t$.  Let us write $u$ as a superposition of eigenfunctions of the Laplacian with
Neumann boundary condition,
and let it evolve according to a classical random walk, but in such a way that spectral components relative to high frequencies move slower than
the ones relative to low frequencies. More precisely, the time step $\tau_k$ is chosen to be larger if the frequency is
higher, according to the following relation:
$$\tau_k=\lambda_k^{1-s}h_k^2,$$
where $\lambda_k$ denotes the Neumann eigenvalue and $h_k$ the space step.
The fact of having homogeneous Neumann boundary conditions, means that the process get reflected when it reaches the boundary.

Another possible notion of nonlocal Neumann Laplacian, is the one introduced in \cite{DROV}, where the authors defined the following  \textit{nonlocal normal derivative}:
\[\mathcal N_s u(x):=c_{n,s} \int_\Omega\frac{u(x)-u(y)}{|x-y|^{n+2s}}\,dy.\]
With this definition, the problem
\begin{equation}\label{NonNeu}
\begin{cases}
(-\Delta)^s u=f & \mbox{in}\;\; \Omega\\
\mathcal N_s u = 0 &\mbox{in}\;\; \mathbb R^n\setminus \overline{\Omega},
\end{cases}
\end{equation}
has a variational structure (here $(-\Delta)^s$ denotes the usual fractional Laplacian in $\mathbb R^n$). Moreover, as shown in \cite[Section 5]{DROV}, $\mathcal N_s u$ recovers the classical normal derivative of $u$ in the limit $s\rightarrow 1$. Another interesting feature of problem \eqref{NonNeu} is that it  can be reformulated through a regional operator with a kernel having logarithmic behaviour at the boundary.
Such notion has also a probabilistic interpretation, as described in Section 2 of \cite{DROV}. 

More precisely,  let us a consider a particle that moves randomly in $\R^n$ according to the following law: a particle situated at  a point $x\in \R^n$, can jump at any other point $y \in \R^n$ with a probability that is proportional to $|x-y|^{-n-2s}$. It is well known that the probability density $u(x,t)$ that the particle is located at the point $x$ at time $t$, solves the fractional heat equation $u_t + (-\Delta)^s u =0$. 
If we consider now such a process in a bounded domain $\Omega$, imposing homogeneous Neumann condition $\mathcal N_s u=0$ corresponds to the fact that if a particle reaches a point $x \in \R^n \setminus \overline \Omega$, it may jump back at any point $y \in \Omega$ with a probability density proportional to $|x-y|^{-n-2s}$. 

Another possibility is to consider the \textit{regional fractional Laplacian}, which corresponds, from a probabilistic point of view, to a censored process: in this case the particle cannot exit the set $\Omega$. As observed in \cite{DROV}, it is possible to define a natural notion of homogeneous Neumann boundary condition for this kind of operator which still gives a variational structure to the problem. However, one cannot consider nonhomogeneous conditions in this setting.

Finally, in \cite{Grubb1} a further notion of fractional Neumann Laplacian was considered, in which the classical normal derivative is replaced by a "weighted" classical derivative of the form $\partial_\nu(u/d^{s-1})$, where $d$ denotes the distance from the boundary of $\Omega$.  
}

The paper is organized as follows:
\begin{itemize}
	\item In Section 2, we introduce the functional setting and the extension problem for the spectral Neumann fractional Laplacian;
	\item Section 3 deals with regularity results for weak solutions;
	\item In Section 4, we prove the existence result Theorem \ref{existence result intro};
	\item  In Section 5, we prove the uniform $L^\infty$ estimate of Theorem \ref{uniformly bounded-intro} and, as a consequence, the non--existence result Theorem \ref{non-existence-intro};
	\item Finally,  Section 6 contains some properties of the non--constant solutions given by Theorem \ref{existence result intro}. In particular, we show that they concentrate at some points (see Theorem \ref{upper level}). A crucial ingredient will be an Harnack inequality contained in \cite{cabre2014nonlinear}.

\end{itemize}

\medskip
\textbf{Notations:} Let $\mathbb R^{n+1}=\{(x,y)\in \mathbb R^n \times \mathbb R\}$ and $\mathbb R_+^{n+1}=\{(x,y)\in \mathbb R^{n+1} \,:\, y>0\}$. We denote by $B(x_0,r)$ the ball of radius $r$ centered at $x_0$  in $\mathbb R^n$, and we set $B_r:=B(0,r)$. Moreover, $\widetilde B((x_0,y_0),r)$ denotes the ball of radius $r$ centered at $(x_0,y_0)\in \mathbb R^{n+1}$ and we set $\widetilde B^+((x_0,y_0),r)=\widetilde B((x_0,y_0),r)\cap\mathbb R_+^{n+1}.$ Finally, we will denote with $\widetilde B^+_r=\widetilde B^+((0,0),r).$

\section{Functional Setting}
Let $\{\varphi_k\}_{k\in \mathbb N} $ be an Hilbert basis of $L^2(\Omega)$ made by eigenfunctions of $-\Delta$ with Neumann boundary condition and let $\{\lambda_k\}_{k\in \mathbb N} $ be the corresponding eigenvalues, i.e.:
\[\begin{cases}
-\Delta \varphi_k=\lambda_k \varphi_k & \mbox{in } \Omega\\
	\partial_\nu \varphi_k=0 & \mbox{on } \partial \Omega.
\end{cases}
\]
%

 It is known that $\lambda_0=0,$ $\int_\Omega\varphi_k=0$ and $\varphi_k\in C^{\infty}(\overline{\Omega})$  for any $k\in \mathbb N$. It is easy to see that $\|\varphi_k\|_{H^1(\Omega)}^2=1+\lambda_k$ and that $\{\varphi_k\}$ are also an orthogonal basis of $H^1(\Omega).$ Therefore we can characterize $H^1(\Omega)$ as \[
H^1(\Omega)=\big\{u\in L^2(\Omega)\::\:\|u\|_{H^1(\Omega)}^2=\sum_{k=0}^{+\infty}(1+\lambda_k)|u_k|^2<+\infty\big\},
\]where $u_k:=\langle u;\varphi_k\rangle_{L^2(\Omega)}.$ 

For any $u\in H^1(\Omega)$, we can define $-\Delta_N u$ (the Neumann Laplacian) as an element of the dual space $H^1(\Omega)'$ in the following way: for any function $v\in H^1(\Omega)$ the action of $-\Delta_N u$ is given by
 \[
\langle-\Delta_N u;v\rangle=\sum_{k=0}^{+\infty}\lambda_k u_kv_k.
\]
In the same way, we can define the natural domain of the $\varepsilon$--fractional Neumann Laplacian $(-\varepsilon \Delta_N)^s$ as the space 
\[
\mathcal{H}^s_{\varepsilon}(\Omega):=\big\{u \in L^2(\Omega): \sum_{k=1}^{\infty}\left(\varepsilon \lambda_k\right)^s| u_k|^2<\infty\big\}
.\]

{ 
Notice that $\mathcal{H}^s_{\varepsilon}(\Omega)$ endowed with the scalar product
$$\langle  u;v\rangle_{	\mathcal{H}^s_{\varepsilon}}:=\langle  u;v\rangle_{L^2(\Omega)} +  \sum_{k=1}^{\infty}\left(\varepsilon \lambda_k\right)^s u_kv_k,$$
and the associated norm

$$\|u\|_{	\mathcal{H}^s_{\varepsilon}}:=\|u\|_{L^2(\Omega)} +  \bigg(\sum_{k=1}^{\infty}\left(\varepsilon \lambda_k\right)^s |u_k|^2\bigg)^\frac12$$
is a Hilbert space.
}

As before, we see  $(-\varepsilon \Delta_N)^s$ as an element of $\mathcal{H}^s_{\varepsilon}(\Omega)'$ whose action is given by
 \begin{equation}
\label{fractional L}\langle(-\varepsilon\Delta_N)^s u;v\rangle=\sum_{k=0}^{+\infty}(\varepsilon\lambda_k)^su_kv_k\quad \mbox{for any } v \in \mathcal{H}^s_{\varepsilon}(\Omega)
.\end{equation}

{
	\begin{observation}\label{spaces} Let $H^s(\Omega)$ be the usual fractional Sobolev space, defined as the closure of $C^{\infty}(\Omega)$ with respect to the norm 
		$$\|u\|_{H^s(\Omega)}:=\|u\|_{L^2(\Omega)}+[u]_{H^s(\Omega)}=\|u\|_{L^2(\Omega)}+\left(\iint_{\Omega \times \Omega}\frac{|u(x)-u(\bar x)|^2}{|x-\bar x|^{n+2s}}\,dx\,d\bar x\right)^{\frac{1}{2}}.$$
		
		It was shown in \cite[Lemma 7.1]{caffarelli2016fractional} that\footnote{In \cite[Lemma 7.1]{caffarelli2016fractional} they require the functions to have $0$ average, however this is not needed in the proof.}
		\begin{equation}\label{equivalence1}\mathcal{H}^{\varepsilon,s}(\Omega)=H^s(\Omega)
		\end{equation} and  
	\begin{equation}\label{equivalence}
			\|\left(-\varepsilon\Delta_N\right)^{\frac{s}{2}} u\|^2_{L^2(\Omega)} \sim\varepsilon^{s}[u]_{H^s(\Omega)}^2.
		\end{equation}
		
		\end{observation}
	
	As explained in the Introduction, we will consider an extension problem in the half--space $\mathbb R^{n+1}_+:=\{(x,y)\in \mathbb R^{n+1}\, |\, x \in \mathbb R^n,\, y>0\}$, associated to \eqref{main system}. To this aim, let us introduce a suitable weighted Sobolev space.
	Let $\mathcal C:=\Omega \times \R_+$ and let $a=1-2s$; we define 
	\[H^1(\mathcal C, y^{a}):=\left\{V\in L^1_{\rm{loc}}(\mathcal C)\,|\, y^{{a}}(|V|^2+|\nabla V|^2) \in L^1(\mathcal C)\,\right\}.\]
}


In \cite{stinga2010fractional,stinga2010extension} it was shown that is possible to realize the spectral fractional Neumann Laplacian as a Dirichlet to Neumann map, as the following Theorem states:

\begin{thm}[\cite{stinga2010fractional,stinga2010extension}]\label{thm-extension}
    For every $u\in\mathcal{H}^s_\varepsilon(\Omega)$ with $\int_\Omega u(x)=0$ there exists a unique $U\in {H}^1(\mathcal{C},y^a)$ such that $\int_{\Omega} U(x,y)dx=0$ for all $y>0$ and \begin{equation}\label{extension}
    \begin{cases}
        \displaystyle \varepsilon \Delta_x U+\frac{a}{y}U_y+U_{yy}=0&\quad\text{in $\mathcal{C}$}\\
        U(x,0)=u(x)&\quad\text{in $\Omega$}\\
        \partial_\nu U=0&\quad\text{on $\partial_L\mathcal{C}$}.
    \end{cases}
    \end{equation}
    
    
    The function $U$ is given by
 \begin{equation}\label{extension semigroup}
    U(x, y)=\sum_{k=1}^{\infty} \rho(\varepsilon^{1/2}\lambda_k^{1 / 2} y) u_k \varphi_k(x)
    \end{equation}
where $\rho(t)$ is the unique solution of 
\[
    \left\{\begin{array}{l}
\rho^{\prime \prime}(t)+\frac{1-2 s}{t} \rho^{\prime}(t)=\rho(t) \quad t>0 \\
-\displaystyle{\lim _{t \to 0^{+}} t^{1-2 s} \rho^{\prime}(t)=c_s} \\
\rho(0)=1,
\end{array}\right.
    \]
    Here $c_s$ is a positive constant depending on $s$.

     Moreover, we have 
     \begin{equation}\label{realisation of fractional laplacian}
    c_{s}(-\varepsilon\Delta_N)^su(x)=-\lim_{y\to 0+}y^aU_y(x,y).
    \end{equation}
Finally, $U$ is the unique minimizer of the functional \begin{equation}\label{energy functional}
    \mathcal{F}(U)=\frac{1}{2} \iint_{\mathcal{C}_{}} \left(\varepsilon\left|\nabla_x U\right|^2+\left|U_y\right|^2\right) y^ad x d y
    \end{equation}among all functions of $H^1(\mathcal{C},y^a)$ whose trace over $\Omega$ is $u$.
\end{thm}
\begin{observation}\label{epsilon}
	\begin{itemize}
		\item {The explicit expression of the solution $U_1$ to problem \eqref{extension} for $\varepsilon=1$, can be found, for example, in \cite[Lemma 3.4 ]{BCdPS}. By a simple scaling argument, one can see that $U(x,y):=U_1(x,\varepsilon^{1/2}y)$ solves \eqref{extension}.
			
		}
    \item The previous result states that it is possible to realize the spectral fractional Laplacian as a Dirichlet to Neumann operator provided $u_\Omega:=\frac{1}{|\Omega|}\int_\Omega u=0$. To give a suitable definition of extension also for functions that do not have zero average, it is enough to consider the extension $\widetilde U$ of $u-u_\Omega$ and define $U\colon=\widetilde U+u_\Omega.$ Clearly both $U$ and $\widetilde U$ solve the same equation with Neumann boundary condition on $\partial_L\mathcal{C}$, and we have \[
    c_{s}(-\varepsilon\Delta_N)^su=c_{s}(-\varepsilon\Delta_N)^s\widetilde u=-\lim_{y\to 0+}y^a\widetilde U_y(x,y)=-\lim_{y\to 0+}y^aU_y(x,y)
    .\]
    \item Observe that, for the sake of readibility, in our problem \eqref{main system}, we are considering $c_s=1$, which is not restrictive up to renaming $\varepsilon$.
    \end{itemize}
\end{observation}
Let us define now the space $\mathrm{H}^{\varepsilon}(\mathcal{C},y^a)$ as the completion of $H^1(\mathcal{C},y^a)$ with respect to
the norm
$$\|U\|^2_{\varepsilon,a}:=\iint_{\mathcal{C}}\left(\varepsilon^{}|\nabla_xU|^2+|U_y|^2\right)y^a dxdy+\int_{\Omega}|U(x,0)|^2dx.$$

Notice that $\mathrm{H}^{\varepsilon}(\mathcal{C},y^a)$, endowed with the scalar product \[
(v,w)_{\varepsilon,a}=\iint_{\mathcal{C}}\left(\varepsilon^{}\nabla_xU\cdot \nabla_xV+U_yV_y\right)y^a dxdy+\int_{\Omega}U(x,0)V(x,0)dx,\]
is a Hilbert space. Clearly, the extension $U$ defined in the previous observation belongs to this space. 

{ We observe that constant functions belong to $\mathrm{H}^{\varepsilon}(\mathcal{C},y^a)$ but not to $H^1(\mathcal C, y^a)$, hence
	 the inclusion 
$$H^1(\mathcal C, y^a) \subset \mathrm{H}^{\varepsilon}(\mathcal{C},y^a)$$
is strict.}
{This can be seen easily in the following way: without loss of generality, let $V\equiv 1$, and let, for every $n\in \mathbb N$,  $$V_n(x,y)=\max\left\{0, 1-\frac{y}{n}\right\}\quad \mbox{for } (x,y)\in \mathcal C.$$
Then, obviously $V_n\in H^1(\mathcal C,y^a)$, $V_n(\cdot ,0)=V(\cdot,0)$ on $\Omega$, moreover
$$\iint_\mathcal C \big(\varepsilon|\nabla_x V_n|^2 + |\partial_y V_n|^2\big)y^a \,dx\,dy=\frac{|\Omega|}{n^2} \int_0^n y^{1-2s}\,dy=\frac{|\Omega|}{2-2s}\frac{1}{n^{2s}}.$$
Hence, 
$$\|V_n-V\|_{\varepsilon,a}\longrightarrow 0\quad \mbox{as } n\rightarrow +\infty.$$
}


{We recall that, by a result of Nekvinda \cite{Nekvinda} (see also \cite{lions1961problemes}), we know that there exists a linear and bounded trace operator 
	\begin{equation*}T^{ a}: H^1(\mathcal C, y^a)\rightarrow H^s(\Omega).\end{equation*}
	
	By Remark \ref{spaces}, we have that
	%
	\begin{equation}\label{trace-operator}T^{ a}: H^1(\mathcal C, y^a)\rightarrow \mathcal H^{\varepsilon,s}(\Omega).\end{equation}
For simplicity of notation, in the following we will denote by $V(\cdot, 0)$, the trace over $\Omega$  of a function $V\in H^1(\mathcal C, y^a)$.

In the following Proposition, we extend the trace operator defined on $H^1(\mathcal C, y^a)$ to the larger space $\mathrm{H}^{\varepsilon}(\mathcal{C},y^a).$ }
\begin{prp}\label{norma}
    For all $V\in H^1(\mathcal{C},y^a)$, there exists a constant $C_s$ independent on $\varepsilon$ such that 
    \[\begin{split}
\left\|\left(-\varepsilon \Delta_N\right)^{\frac{s}{2}} V(\cdot, 0)\right\|_{L^2(\Omega)}^2  &=\sum_{k=1}^{\infty}\left(\varepsilon \lambda_k\right)^{s}\left|\langle V(\cdot,0); \varphi_k\rangle_{L^2(\Omega)}\right|^2 \\
&
\leq C_s\iint_{\mathcal{C}}\left(\varepsilon\left|\nabla_x V\right|^2+\left|V_y\right|^2\right)y^a d x d y ,
\end{split}\]
where the equality holds if and only if $V$ solves problem (\ref{extension}).

Moreover, there exists a unique bounded linear operator $T^{\varepsilon,a}:\mathrm{H}^{\varepsilon}(\mathcal{C},y^a)\to\mathcal H^{\varepsilon,s}(\Omega)$, whose operator norm is proportional to $\varepsilon^{\frac{s}{2}}$, such that $T^{\varepsilon,a}V(\cdot)=V(\cdot,0)$ for any $V\in H^1(\mathcal{C},y^a)$.
\label{trace}\end{prp}
\begin{proof}
    The proof is an adaptation of \cite[Lemma 2.4]{stinga2015fractional}. Let $u(\cdot)=V(\cdot,0)$ be the trace of $V$ over $\Omega$. Without loss of generality, we may assume $\int_\Omega u=0$. Let $U$ be the solution of problem (\ref{extension}) with $U(\cdot,0)=u(\cdot)$ over $\Omega$. Since $U$ is a minimizer of the functional (\ref{energy functional}) among all functions having the same trace over $\Omega$, we have 
    \[
    \iint_{\mathcal{C}}\left(\varepsilon|\nabla_x U|^2+| U_y|^2\right)y^a d x d y \le\iint_{\mathcal{C}}\left(\varepsilon|\nabla_x V|^2+|V_y|^2\right)y^a d x d y 
    .\]Using the explicit formula (\ref{extension semigroup}) for $ U$ and recalling that each $\varphi_k$ has zero average we find 
    \[
    \begin{split}
   & \iint_{\mathcal{C}}\left(\varepsilon|\nabla_x  U|^2+| U_y|^2\right)y^a d x d y\\
   &\hspace{2em}=\int_0^{+\infty} \left (\sum_{k=1}^{+\infty} \rho^2\left((\varepsilon \lambda_k)^{1/2} y\right) u_k^2 \varepsilon \lambda_k +\sum_{k=1}^{+\infty} \dot{\rho}^2\left((\varepsilon \lambda_k)^{1/2} y\right) u_k^2\varepsilon \lambda_k\right) y^a d y. \end{split}
    \]
    Performing that change of variable $(\varepsilon \lambda_k)^{1/2}y=t$ and recalling that $a=1-2s$ we find 
    \[
    \begin{split}
    \iint_{\mathcal{C}}\left(\varepsilon|\nabla_x  U|^2+| U_y|^2\right)y^a d x d y&=\left(\int_0^{+\infty}\left(\rho^2(t)+\dot{\rho}^2(t)\right) t^{1-2s} d t  \right)\cdot\left(\sum_{k=1}^{+\infty}(\varepsilon \lambda_ k)^s u_k^2\right)\\
    &=C_s\left\|(-\varepsilon \Delta)^{\frac{s}{2}} u\right\|_{L^2(\Omega)}^2,
    \end{split}
    \]where 
    \[
    C_s=\int_0^{+\infty}\left(\rho^2(t)+\dot{\rho}^2(t)\right) t^{1-2s} d t.
    \]Hence the inequality is established.

    	  Finally, the trace operator $T^{\varepsilon, a}$ in \eqref{trace-operator}  can be extended by density from $H^1(\mathcal C, y^a)$ to $\mathrm{H}^{\varepsilon}(\mathcal{C},y^a)$.
    	  
\end{proof}

\medskip
{\begin{observation}\label{compactness-trace}Using Proposition \ref{trace}, the fact that $\mathcal H^{\varepsilon,s}(\Omega)=H^s(\Omega)$ (see \eqref{equivalence1}-\eqref{equivalence}) and the fractional Sobolev embedding, we have for any $V\in  \mathcal H^{\varepsilon,s}(\Omega)$
		\begin{equation}\label{trace-Sobolev}\left\|V(\cdot,0)\right\|_{L^{\frac{2n}{n-2s}}(\Omega)}\le C\left\|V(\cdot,0)\right\|_{H^s(\Omega)}\le \frac{C}{\varepsilon^{s/2}}\left\|V\right\|_{\varepsilon,a}.\end{equation}

		 As a consequence, the operator $T^{\varepsilon,a}$ is compact from $\mathrm{H}^\varepsilon(\mathcal{C},y^a)$ to $L^q(\Omega)$ for every $1\le q<\frac{2n}{n-2s}.$
	\end{observation}}

\section{Regularity of solutions}In this section we want to study the regularity of solutions to \begin{equation}\label{general force}
(-\varepsilon \Delta_N)^{s} u+u=f\quad \text{ in }H^{-s}(\Omega).
\end{equation}First of all, we notice that the unique solution of (\ref{general force}) is given by \begin{equation}\label{explicit solution}
u(x)=\sum_{k=1}^{+\infty}\frac{1}{1+(\varepsilon\lambda_k)^s}f_k\varphi_k(x)
\end{equation} where $f_k=\langle f;\varphi_k\rangle$. Now for every $u\in L^2(\Omega)$ we define the fractional heat semigroup as 
\begin{equation}\label{fractional semigroup}
e^{-t(-\varepsilon\Delta_N)^s}u(x)=\sum_{k=1}^{\infty}e^{-t(\varepsilon\lambda_k)^st}u_k\varphi_k(x).
\end{equation}

Thanks to this definition, we can rewrite (\ref{explicit solution}) as \begin{equation}\label{u integrale}
u(x)=\int_0^{+\infty} e^{-t-t (-\varepsilon\Delta_N)^s} f(x) d t.
\end{equation}{Let $s\in (0,1)$ and $u\in\mathcal H^{s,\varepsilon}(\Omega)$ be a function with zero average, we define the function $v_s$ as $$v_s(x,t):=e^{-t(-\varepsilon\Delta_N)^{s}}u(x).$$ 
If $s={1}/{2}$, we can observe that $v:=v_{{1}/{2}}$ is the unique solution of  \[
\begin{cases}\varepsilon \Delta_x v+v_{t t}=0, & \text { in } \mathcal{C} \\ \partial_\nu v=0, & \text { on } \partial_L \mathcal{C} \\ v(x, 0)=u(x), & \text { on } \Omega,\end{cases}
\]see for instance \cite[Theorem 2.1]{stinga2015fractional}. Hence, one can write $v_{1/2}$ as a convolution between $u$ and a kernel $P_t(x)$ that behaves like \[
P_{t}(x)\sim \frac{ t}{\left( t^2+|x|^2\right)^{\frac{n+1}{2}}} \quad \text{if $x\in \Omega$ and $0<t<1$,}
\]and is uniformly bounded in $x$ if $t>1$. Thanks to this information and \eqref{u integrale} it is possible to obtain an integral representation of $u$ in terms of $f$ and $P_{\sqrt\varepsilon t}(x)$. Such representation was used in \cite[Theorem 3.5]{stinga2015fractional} to obtain regularity results for $u.$  If $s\neq\frac{1}{2}$ it is not clear if $v_s$ solves an elliptic equation in $\mathcal C$ and therefore if an integral representation of $u$ in terms of $f$ and a suitable kernel even exists. Hence, in order to obtain some regularity results, we decided to follow an approach based only on semigroup theory. }
 
It is well known (see for example \cite[{Theorem 1.3.2. and Theorem 1.3.3.}]{Davies}), that
the fractional Neumann heat semigroup is contractive from $L^p(\Omega)$ to $L^p(\Omega)$ for every $p\in[1,\infty]$, namely \begin{equation}\label{L^p-L^p}
\left\|e^{-t\left(-\varepsilon \Delta_N\right)^s} u\right\|_{L^{p}(\Omega)} \leq\|u\|_{L^p(\Omega)}.
\end{equation}Moreover, since the first eigenvalue of the Neumann Laplacian vanishes, we 
have  that  the following $L^\infty-L^1$ estimate holds true (see e.g.  \cite[Theorem 2.10 (b),  Remark 2.11]{warna} and references therein):
  \begin{equation}\label{1-infinito}
\left\|e^{-t(-\varepsilon\Delta_N)^s}u\right\|_{L^\infty(\Omega)}\leq C_\varepsilon e^{t}t^{-\frac{n}{2s}}\|u\|_{L^1(\Omega)}\quad\text{for all $t>0$},
\end{equation}here $C_\varepsilon$ is a constant that depends on $s$, $n$, $\Omega$ and $\varepsilon.$
In order to obtain an $L^q-L^p$ estimate we need to recall the statement of the Riesz--Thorin interpolation Theorem, for a proof see for instance \cite[Theorem 2.1]{stein2011functional}.\begin{thm}\label{riesz--thorin}
    Let $(X,\mu)$ and $(Y,\nu)$ be a pair of $\sigma-$finite measure space, and let $1\leq p_0,q_0,p_1,q_1\leq+\infty.$ If $$T:L^{p_0}(X,\mu)+ L^{p_1}(X,\mu)\mapsto L^{q_0}(Y,\nu)+ L^{q_1}(Y,\nu)$$ is a bounded and linear operator, i.e. there exist constants $M_0$ and $M_1$ such that \[
    \begin{cases}
\|T(f)\|_{L^{q_0}(Y)} \leq M_0\|f\|_{L^{p_0}(X)} \\
\|T(f)\|_{L^{q_1}(Y)} \leq M_1\|f\|_{L^{p_1}(X)},
\end{cases}
    \]then for every $p$, $q$, and $t\in[0,1]$, that fulfill the relations \[
\frac{1}{p}=\frac{1-t}{p_0}+\frac{t}{p_1} \quad \text { and } \quad \frac{1}{q}=\frac{1-t}{q_0}+\frac{t}{q_1}
,    \]there exists a constant $M$ such that \[
    \|T(f)\|_{L^{q}(Y)} \leq M\|f\|_{L^{p}(X)}.
    \]Moreover the constant $M$ satisfies $M\leq M_0^{1-t}M_1^t.$
\end{thm}
We can now show the following
\begin{prp}
   For every $p\in[1,+\infty]$ and $q\in[p,+\infty]$, it holds\begin{equation}\label{ultracontractivity}
\left\|e^{-t(-\varepsilon \Delta_N)^s} u\right\|_{L^{q}(\Omega)} \leq C_{\varepsilon}^{\left(\frac
1 p-\frac 1 q\right)} e^{t\left (\frac 1 p -\frac 1 q\right)} t^{-\frac{n}{2 s}\left(\frac 1 p-\frac 1 q\right)}\|u\|_{L^p(\Omega)}
\end{equation}where $C_\varepsilon$ is the same constant of (\ref{1-infinito}) and we adopt the convention that $0=\frac{1}{\infty}.$  
\end{prp}
\begin{proof}
    Let us denote by $$\left\|e^{-t(-\varepsilon\Delta_N)^s}\right\|_{\mathcal{L}(L^p,L^q)}$$ the operator norm of $e^{-t(-\varepsilon\Delta_N)^s}$ as an operator from $L^p(\Omega)$ to $L^q(\Omega)$. From Theorem \ref{riesz--thorin} and estimates \eqref{L^p-L^p} and \eqref{1-infinito} we have (writing $\frac{1}{p}=\frac{t}{1}+\frac{1-t}{\infty}$ with $t=\frac{1}{p}$)\[
    \left\|e^{-t(-\varepsilon\Delta_N)^s}\right\|_{\mathcal{L}(L^p,L^\infty)}\leq\left\|e^{-t(-\varepsilon\Delta_N)^s}\right\|_{\mathcal{L}(L^1,L^\infty)}^{\frac{1}{p}}\left\|e^{-t(-\varepsilon\Delta_N)^s}\right\|_{\mathcal{L}(L^\infty,L^\infty)}^{1-\frac{1}{p}}\leq C_\varepsilon^{\frac{1}{p}}e^\frac{t}{p}t^{-\frac{n}{2sp}}.
    \]Now, for any $q\geq p$ we have, again by Theorem \ref{riesz--thorin} and estimates \eqref{L^p-L^p} and \eqref{1-infinito} (this time we write $\frac{1}{q}=\frac{t}{p}+\frac{1-t}{\infty}$ with $t=\frac p q)$ \begin{align*}
    \left\|e^{-t(-\varepsilon\Delta_N)^s}\right\|_{\mathcal{L}(L^p,L^q)}\leq& \left\|e^{-t(-\varepsilon\Delta_N)^s}\right\|_{\mathcal{L}(L^p,L^p)}^{\frac p q}\left\|e^{-t(-\varepsilon\Delta_N)^s}\right\|_{\mathcal{L}(L^p,L^\infty)}^{1-\frac{p}{q}}\\\leq& \left(C_\varepsilon^{\frac{1}{p}}e^\frac{t}{p}t^{-\frac{n}{2sp}}\right)^{1-\frac p q}\\=& C_{\varepsilon}^{\left(\frac
1 p-\frac 1 q\right)} e^{t\left (\frac 1 p -\frac 1 q\right)} t^{-\frac{n}{2 s}\left(\frac 1 p-\frac 1 q\right)},
    \end{align*}that is equivalent to the desidered estimate.
\end{proof}
\medskip
{Interior and boundary regularity for solutions of the equation
$$(-\Delta)^su=f,$$
have been established in \cite{caffarelli2016fractional} for both the cases of 0--Dirichlet and 0--Neumann boundary conditions. In particular, regularity in H\"older spaces has been considered.

Here, we start by considering regularity in $L^q$ spaces for our equation \eqref{general force}, extending the results in \cite[Theorem 3.5]{stinga2015fractional} to any fractional power $s\in(0,1)$.
}

\begin{thm}\label{first regularity}
    Let $u$ be a solution of (\ref{general force}). Then the following facts hold:\begin{enumerate}[(i)]
        \item If $f\in L^p$ with $1\le 2sp< n$ then $u\in L^q(\Omega)$ for every $p\le q< p^*$ where $p^*=\frac{np}{n-2sp}$ and \[\|u\|_{L^q(\Omega)} \leq C_{\varepsilon, n, p, q,s, \Omega}\|f\|_{L^p(\Omega)}.\]
        \item If $f\in L^p(\Omega)$ with $2sp>n$ then $u\in L^\infty(\Omega)$ and \[
        \|u\|_{L^{\infty}(\Omega)} \leq C_{\varepsilon, n, p,s, \Omega}\|f\|_{L^p(\Omega)} .
        \]
    \end{enumerate}
\end{thm}
\begin{proof}
    If $f\in L^p(\Omega)$ then we can write \[
    u(x)=(I+(-\varepsilon\Delta_N)^s)^{-1}f(x)=\int_{0}^{+\infty}e^{-t-(t\varepsilon^s)(-\Delta_N)^s}f(x)dt.
    \]If $2sp<n$ and $q\in[p,\frac{np}{n-2sp})$ we can use (\ref{ultracontractivity}) to infer that 
    \begin{equation}\label{regularity_comp}
    \begin{split}
\|u\|_{L^q(\Omega)}& \leq \int_0^{+\infty} e^{-t}\left\|e^{-t(-\varepsilon \Delta)^s} f\right\|_{L ^q(\Omega)} d t \\
&
 \leq C_\varepsilon^{\left(\frac{1}{p}-\frac{1}{q}\right)}\|f\|_{L^p(\Omega)}\int_0^{+\infty}  e^{-t\left(1-\left(\frac{1}{p}-\frac{1}{q}\right)\right)} t^{-\frac{n}{2 s}\left(\frac{1}{p}-\frac{1}{q}\right)}dt
 \end{split}
\end{equation} 
where $C_\varepsilon$ is a constant that depends on $s$, $n$, $\Omega$, and $\varepsilon$. The conclusion follows observing that the last integral in (\ref{regularity_comp}) converges if and only if $q< \frac{np}{n-2sp}.$
Assume now $2sp>n$, thanks again to \eqref{ultracontractivity} we have \begin{equation}\label{regularity-computation2}\|u\|_{L^\infty(\Omega)} \leq C_\varepsilon^{\frac{1}{p}}\|f\|_{L^p(\Omega)} \int_0^{+\infty} e^{-t\left(1-\frac{1}{p}\right)} t^{-\frac{n}{2 sp}}dt,\end{equation}where $C_\varepsilon$ is again a constant that depends on $s$, $n$, $\Omega$, and $\varepsilon$. Again the conclusion follows observing that the integral in (\ref{regularity-computation2}) converges if and only if $2sp>n.$
\end{proof}
\begin{thm}\label{second regularity}
    Let $u$ be a solution of (\ref{general force}) with $f\in L^p(\Omega)$, we have: \begin{enumerate}[(i)]
        \item If $2sp>n$ let $\alpha=2s-\frac{n}{p}$.
        \begin{itemize}
            \item If $\alpha<1$ then $u\in C^{0,\alpha}(\overline{\Omega})$;
\item If $\alpha=1$ then $u\in C^1(\overline{\Omega});$ \item If $1<\alpha<2$ then $u\in C^{1,\alpha-1}(\overline{\Omega})$. 
\end{itemize}
{In all three cases the corresponding seminorm is bounded by a constant (that depends on $s$, $\varepsilon$, $n$, $\alpha$ and $\Omega$) times the norm of $f$ in $L^p(\Omega).$}
        \item If $f\in L^\infty(\Omega)$, we have:\begin{itemize}\item if $s<\frac{1}{2}$ then $u\in C^{0,2s}(\overline{\Omega});$\item if $s=\frac{1}{2}$ then $u\in C^1(\overline{\Omega})$;\item  if $s>\frac{1}{2}$ then $u\in C^{1,2s-1}(\overline{\Omega})$.
   \end{itemize} {In all three cases the corresponding seminorm is bounded by a constant (that depends on $s$, $\varepsilon$, $n$ and $\Omega$) times the norm of $f$ in $L^\infty(\Omega)$.}\end{enumerate}
\end{thm}
\begin{proof}
   In booth cases $(i)$ and $(ii)$ we can use Theorem \ref{first regularity} to infer that $u\in L^\infty(\Omega)$, hence $(-\varepsilon\Delta_N)^su$ belongs to the same $L^p$ space as $f$, now the conclusion follows by \cite[Corollary 4.3]{Grubb}\footnote{{Even if the controls on the seminorms of $u$ in term of the $L^p$ norms of $f$ are not explicitely stated in \cite[Corollary 4.3]{Grubb}, a careful reading of the proof shows that they hold true.}}.
\end{proof}
{ The following Theorem contains regularity results in H\"older spaces, which follows easily from the ones in \cite{caffarelli2016fractional}.}
\begin{thm}\label{regularity 3}
   Let $u$ be a solution of (\ref{general force}) with $f\in C^{0,\alpha}(\overline{\Omega})$, we have :\begin{enumerate}[(i)]
       \item If $0<\alpha+2s<1$,  then $u\in C^{0,\alpha+2s}(\overline{\Omega})$ and $$
       [u]_{C^{0,\alpha+2s}(\overline{\Omega})}\le C\left(\|u\|_{H^s(\Omega)}+\|f\|_{C^{0,\alpha}(\overline{\Omega})}\right).
       $$
       \item If $1<\alpha+2s<2$, 
       then $u\in C^{1,\alpha+2s-1}(\overline{\Omega})$
       \[
       [u]_{C^{1,\alpha+2s-1}(\overline{\Omega})}\le C\left(\|u\|_{H^s(\Omega)}+\|f\|_{C^{0,\alpha}(\overline{\Omega})}\right).
       \]
   \end{enumerate}
\end{thm}\begin{proof}We will only prove $(i)$, since $(ii)$ follows in the same way. {Since $f\in C^{0,\alpha}(\overline{\Omega})$ we can use $(ii)$ of Theorem \ref{second regularity} to infer that 
$u\in C^{0,2s}(\overline{\Omega})$. 


Therefore, if we define $h=f-u$ we have that $h\in C^{0,\min\{2s,\alpha\}}(\overline{\Omega})$ and $u$ solves \[
\begin{cases}
    (-\Delta_N)^su=h\quad&\text{in $\Omega$}\\
    \partial_\nu u=0\quad&\text{on $\partial\Omega$.}
\end{cases}
\]
By the regularity result of \cite[Theorem 1.4]{caffarelli2016fractional} and the triangle inequality we have 
\begin{equation}\label{holder regularity}
 \begin{split}
     [u]_{C^{0,  \min\{2s,\alpha\}+2s}}(\overline{\Omega})&\leq C\left(\|u\|_{H^s(\Omega)}+\|h\|_{C^{0, \min\{2s,\alpha\}}}(\overline{\Omega})\right)\\
 &\leq C\left(\|u\|_{H^{s}(\Omega)}+\|u\|_{C^{0,2s}}(\overline{\Omega})+\|f\|_{C^{0, \alpha}(\overline{\Omega})}\right) .
 \end{split}
\end{equation}
Again by $(ii)$ of Theorem \ref{second regularity} we have \[
\|u\|_{C^{0, 2s}(\overline{\Omega})}\ \leq C\|f\|_{L^{\infty}(\Omega)}.\]
We iterate this procedure a finite number of times $N$ (with $N$ such that $2sN\ge \alpha$) to get the conclusion.}
\end{proof}Thanks to the above results, we can prove that any solution of our nonlinear problem (\ref{main system}) is in $C^{1,\alpha}(\overline{\Omega}).$
\begin{thm}
    \label{golbal holder regularity}Let $u$ be any weak solution of (\ref{main system}),
    then $u\in C^{1,\alpha}(\overline{\Omega})$ for some $\alpha\in(0,1).$
\end{thm}
\begin{proof}
   Thanks to $(ii)$ of Theorem \ref{second regularity}, Theorem \ref{regularity 3} and a standard iterative argument, it is enough to show that $u\in L^\infty(\Omega).$ By the Sobolev embedding we have that $u\in L^q(\Omega)$ where $q=\frac{2n}{n-2s},$ hence $u^p\in L^{\frac{q}{p}}(\Omega).$ Assume now that $u\in L^r(\Omega)$ for some $r\ge q$. Then $u^p\in L^{\frac{r}{p}}(\Omega)$ and we have three cases.
   If $\frac{2sr}{p}>n$ we can use $(ii)$ of Theorem \ref{first regularity} to obtain that $u\in L^\infty(\Omega)$. If $\frac{2sr}{p}=n$, since $\Omega$ is bounded we can find  $\eta>0$ such that $u^p\in L^{\frac{r}{p}-\eta}(\Omega)$. Using $(i)$ of Theorem \ref{first regularity} we find that $u^p\in L^{\widetilde q}(\Omega)$ for every $\widetilde q<\frac{\left(\displaystyle{\frac{r}{p}}-\eta\right)^*}{p}$, where
\[  \left(\frac{r}{p}-\eta\right)^*= \frac{n\left(\displaystyle{\frac{r}{p}}-\eta\right)}{n-2 s\left(\displaystyle{\frac{r}{p}}-\eta\right)}.\]

Now we want to choose an $\eta>0$ such that $\frac{\left(\frac{r}{p}-\eta\right)^*}{p}>\frac{n}{2s}$, a straightforward computation shows that is enough to choose 
{$\eta<\frac{n}{2s(p+1)}$.}

Hence we can use again $(ii)$ of Theorem \ref{first regularity} to obtain that $u\in L^\infty(\Omega).$ If $\frac{2sr}{p}<n$ then $u\in L^\sigma(\Omega)$ for every $\sigma<\left(\frac{r}{p}\right)^*.$ Observe that, since $r\ge q$ we have 
\[\left(\frac{r}{p}\right)^*=\frac{\displaystyle{\frac {nr} {p}}}{n- \displaystyle{\frac{2sr}{p}}}=\frac{n r}{np-2 s r} \geq r\left(\frac{n}{np-2s q}\right):=r\lambda
\]where $\lambda>1$ (since $q=2n/(n-2s)$ and $p<(n+2s)/(n-2s)$). Hence $u\in L^{r\lambda}(\Omega).$ If we iterate this argument $k$ times, we find $u\in L^{r\lambda^k}(\Omega)$, therefore, choosing $k$  as the smallest integer such that $\frac{2s r\lambda^k}{p}\ge n$, we reduce to one of the two previous cases. 
\end{proof}
{The previous regularity result together with \cite[Lemmma 4.5]{cabre2014nonlinear}, implies the following
\begin{cor}\label{u_y}
	Let $u$ be any weak solution of (\ref{main system}) and let $U$ be its extension satisfying \eqref{extension}.
	
	Then, for any half--ball $\widetilde B^+((x_0,0),r) \subset \mathcal C$, we have that $y^a U_y \in C^{0,\beta}(\overline{\widetilde B^+((x_0,0),r)})$, for some $\beta>0$.
	\end{cor}
}
\section{Existence of non--constant solutions for small \texorpdfstring{$\varepsilon$}{}}
In this  Section we establish existence of positive non--constant solutions to problem \eqref{main system} for $\varepsilon$ small enough.

In the sequel, we will adopt the following notation for the positive part of $u$: $u_+=\max\{0,u\}$.

 In order to define a notion of solution for system (\ref{main system}), we consider the following extended problem \begin{equation}\label{extended non linear}
\begin{cases}\varepsilon \Delta_x U+\frac{a}{y}U_y+U_{y y}=0, & \text { in } \mathcal{C}, \\ \partial_\nu U=0, & \text { on } \partial_L \mathcal{C}, \\ -\lim_{y\to 0}y^aU_y(x, y)+U(x, 0)=U(x, 0)_+^p, & \text { in } \Omega.\end{cases}
\end{equation}

\begin{definition}
	\begin{itemize}
		\item We say that a function $U\in \mathrm{H}^\varepsilon(\mathcal{C},y^a)$ is a weak solution of (\ref{extended non linear}) if for every $\Phi\in\mathrm{H}^\varepsilon(\mathcal{C},y^a)$ we have \begin{equation}\label{weak formulation}
\iint_{\mathcal{C}}\left(\varepsilon \nabla_x U \cdot \nabla_x \Phi+U_y \Phi_y\right) y^a d x d y+\int_{\Omega} U(x, 0) \Phi(x, 0) d x=\int_{\Omega}U(x,0)_+^p\Phi(x,0) dx.
\end{equation}

\item  We  say that a function $u\in H^s(\Omega)$ is a weak solution of (\ref{main system}) if its extension in $\mathcal C$ satisfying problem \eqref{extension} is a weak solution of (\ref{extended non linear}).
	\end{itemize}
\end{definition}

{We remark that working in the space $\mathrm{H}^\varepsilon(\mathcal{C},y^a)$ allows us to choose constant functions as test functions in the weak formulation \eqref{weak formulation}.}

It is easy to see that critical points of the functional 
\[
\begin{split}
J_{\varepsilon,a}(U)&:= \frac{\|U\|^2_{\varepsilon,a}}{2}-\frac{1}{p+1}\int_{\Omega}|U(x,0)|^{p+1}dx=\\
&\hspace{1em} \frac{1}{2}\iint_{\mathcal{C}}\left(\varepsilon^{}|\nabla_xU|^2+|U_y|^2\right)y^a dxdy
+\frac{1}{2}\int_{\Omega}|U(x,0)|^2dx -\frac{1}{p+1}\int_{\Omega}|U(x,0)_+|^{p+1}dx
\end{split}
\]
are weak solutions of (\ref{extended non linear}). 

 The following result establishes the existence of non--constant solutions to \eqref{main system} provided the parameter $\varepsilon$ is sufficiently small. 
 
 \begin{thm}\label{existence result}
    If $\varepsilon$ is small enough, there exists at least one non--constant positive weak solution $U^{\varepsilon,a}$ of (\ref{extended non linear})  
    In particular, the function $u^{\varepsilon,a}=U^{\varepsilon,a}(\cdot,0)$ is a weak solution to (\ref{main system}).
    
    Moreover, there exists a positive constant $C=C(n,a,p,\Omega)$, which does not depend on $\varepsilon$, such that $$J_{\varepsilon,a}(U^{\varepsilon,a})\le C\varepsilon^{\frac{n}{2}}.$$ 
\end{thm}\begin{proof}
    The proof is an application of the Mountain Pass Theorem by Ambrosetti and Rabinowitz \cite{ambrosetti1973dual}.
{Similarly to \cite{stinga2015fractional}, one can show that the functional $J_{\varepsilon,a}$ is $C^{1,0}_{\normalfont{\text{loc}}}(\mathrm{H^\varepsilon}(\mathcal{C},y^a),\:\mathbb{R})$ and its derivative $J_{\varepsilon,a}'$ is locally Lipschitz continuous.
	We proceed now as follows.}

    \emph{Step 1: The functional $J_{\varepsilon,a}$ satisfies the Palais--Smale condition.} Let $V_k\in\mathrm{H}^\varepsilon(\mathcal{C},y^a)$ be a sequence such that $\{J_{\varepsilon,a}(V_k)\}_k$ is bounded and $J_{\varepsilon,a}'(V_k)$ tends  to 0 as $k\rightarrow \infty$. We want to show that, up to subsequences, $V_k\to V$ in $\mathrm{H}^\varepsilon(\mathcal{C},y^a)$, for some $V$. Since $J'_{\varepsilon,a}(V_k)$ goes to zero we have \[
   \left|\langle J'_{\varepsilon,a}(V_k); V_k\rangle\right|\le \|V_k\|_{\varepsilon,a}
    \]if $k$ is large enough. Hence, we have 
    %
    %
    {
    	$$\left| \|V_k\|^2_{\varepsilon, a}-\int_\Omega (V_k)_+^{p+1}(x,0)\,dx \right| \leq \|V_k\|_{\varepsilon, a},$$
    }
    and by this expression and the fact that $J_{\varepsilon,a}(V_k)$ is bounded we deduce that \[
    \|V_k\|_{\varepsilon,a}^2\le C+\frac{2}{p+1}(\left\|V_k\right\|_{\varepsilon,a}^2+\left\|V_k\right\|_{\varepsilon,a}).
    \]
    Since $p>1$ we have that the sequence $\{V_k\}_k$ is bounded in $\mathrm{H}^{\varepsilon}(\mathcal{C},y^a)$ and therefore, up to subsequences, it weakly converges to some $V\in \mathrm{H}^{\varepsilon}(\mathcal{C},y^a).$ 
    By the compactness of the trace operator (see Remark \ref{compactness-trace}), we also have $V_k(\cdot,0)\to V(\cdot,0)$ strongly in $L^q(\Omega)$ for every $1\le q<\frac{2n}{n-2s}$ (and thus, in particular, in $L^{p+1}(\Omega)$).
    
    {In order to conclude that such convergence is strong, it is enough to show that
    \begin{equation}\label{convergence-norms}
    \left\|V_k\right\|_{\varepsilon,a}-\left\|V\right\|_{\varepsilon,a}\rightarrow 0.
    \end{equation}

    Again, since $J_{\varepsilon, a}'(V_k)\rightarrow 0$, we also have
    \begin{equation}\label{a}
    \left|\langle J_{\varepsilon, a}'(V_k); V_k-V\rangle \right| \rightarrow 0.
    \end{equation}

    Moroever, using the H\"older inequality, we have that
    \begin{equation}\label{b}
\int_\Omega V_k^p(x,0)(V_k(x,0)-V(x,0))\,dx\le \|V_k(x,0)\|_{L^{p+1}}^p \|V_k(x,0)-V(x,0)\|_{L^p}\rightarrow 0,
    \end{equation}
   and
    \begin{equation}\label{c}
    \begin{split}
&\int_\Omega V_k(x,0)(V_k(x,0)-V(x,0))\,dx\\
&\hspace{2em}=\int_\Omega (V_k(x,0)-V(x,0))^2\,dx + \int_\Omega V(x,0)(V_k(x,0)-V(x,0))\,dx \rightarrow 0,
\end{split}
   \end{equation}
   
    Recalling that
    \begin{equation}
    \begin{split}
        \langle J_{\varepsilon, a}'(V_k); V_k-V\rangle &=\frac{1}{2}\iint_{\mathcal C} y^a(\varepsilon \nabla_x V_k\,,\,\partial_yV_k)\cdot (\nabla(V_k-V))\, dx \, dy \\
        &\hspace{1em} + \int_\Omega V_k(x,0)(V_k(x,0)-V(x,0))\, dx\\
        &\hspace{1em}-\int_\Omega V_k^p(x,0)(V_k(x,0)-V(x,0))\,dx,
        \end{split}
    \end{equation}
and combining \eqref{a}-\eqref{b}-\eqref{c}, we get
 \begin{equation}
     \iint_{\mathcal C} y^a(\varepsilon \nabla_x V_k\,,\,\partial_yV_k)\cdot (\nabla(V_k-V))\, dx \, dy \rightarrow 0.
        \end{equation}

      Observing that
      \begin{equation}
      \begin{split}
&2 \iint_{\mathcal C} y^a(\varepsilon \nabla_x V_k\,,\,\partial_yV_k)\cdot (\nabla(V_k-V))\, dx \, dy\\
&\hspace{2em}\ge  \iint_{\mathcal C} y^a \left(\varepsilon|\nabla_x V_k|^2+|\partial_yV_k|^2\right)-\iint_{\mathcal C}y^a  \left(\varepsilon|\nabla_x V|^2+|\partial_yV|^2\right) \,dx\,dy,
\end{split}
           \end{equation}
      
we deduce that
\begin{equation}\label{limsup}
\iint_{\mathcal C} y^a \left(\varepsilon|\nabla_x V|^2+|\partial_yV|^2\right)\,dx\,dy \ge \limsup_{k\rightarrow \infty} \iint_{\mathcal C} y^a \left(\varepsilon|\nabla_x V_k|^2+|\partial_yV_k|^2\right) \,dx\,dy.
 \end{equation}

The reverse inequality
\begin{equation}\label{liminf}
\iint_{\mathcal C} y^a \left(\varepsilon|\nabla_x V|^2+|\partial_yV|^2\right)\,dx\,dy \le \liminf_{k\rightarrow \infty} \iint_{\mathcal C} y^a \left(\varepsilon|\nabla_x V_k|^2+|\partial_yV_k|^2\right)\,dx\,dy
 \end{equation}

 just follows by lower--semicontinuity.

 Combining \eqref{limsup} and \eqref{liminf}, we obtain \eqref{convergence-norms} and thus conclude that
 $$\|V_k-V\|_{\varepsilon,a}\rightarrow 0.$$

    }

    {\emph{Step 2}: It is standard to verify that the energy functional satisfies the Mounatin pass geometry {(see \cite[Lemma 3.3]{{ambrosetti1973dual}})}, i.e.:
    \begin{itemize}
        \item $J_{\varepsilon,a}(0)=0$;
\item There exists some $\rho>0$ such that 
        $J_{\varepsilon,a}(V)=\beta>0$ if $\|V\|_{\varepsilon,a}=\rho$. This follows by the trace-Sobolev inequality (see Remark \ref{compactness-trace}), which gives that
        $$\int_\Omega |U(x,0)_+|^{{p+1}}\le C \|V\|_{\varepsilon,a}^{p+1}=o(\|V\|_{\varepsilon,a}^2);
        $$
        
        \item there exists $\widetilde V$ with $\|\widetilde V\|_{\varepsilon,a}>\rho$  such that $J_{\varepsilon,a}(\widetilde V)<0.$
        \end{itemize}
}
    \emph{Step 3}: For $\varepsilon>0$ sufficiently small, there is a non negative function $\Phi\in \mathrm{H}^{\varepsilon}(\mathcal{C},y^a)$  and positive constants $t_0$, $C_0$ such that \[
    J_{\varepsilon,a}(t_0\Phi)=0\quad\text{and}\quad J_{\varepsilon,a}(t\Phi)\le C_0\varepsilon^{\frac{n}{2}}\;\;
    \mbox{for all} \;\;t\in [0,t_0].\] 
    
    To prove this, we assume that $0\in \Omega$ and we define the functions \[
    \varphi(x)= \begin{cases}\varepsilon^{-\frac{n}{2}}\left(1-\varepsilon^{\frac{1}{2}}|x|\right) & \text { if }|x|<{\varepsilon^{\frac{1}{2}}} \\ 0 & \text { if }|x| \geq {\varepsilon^{\frac{1}{2}}},\end{cases}
    \]and \[
    \Phi(x, y)=e^{-\frac{y}{2}} \varphi(x).
    \]It is easy to see that $\Phi\in \mathrm{H}^\varepsilon(\mathcal{C},y^a)$ and \[
    \iint_{\mathcal{C}}\left|\nabla_x \Phi\right|^2 y^ad x d y=\Gamma(1+a)\int_{\Omega}|\nabla \varphi|^2 d x, \quad \iint_{\mathcal{C}}\left|\Phi_y\right|^2 y^ad x d y=\Gamma(1+a)\frac{1}{4} \int_{\Omega} \varphi^2 d x,
    \]and \[
    \int_{\Omega}\left|\Phi(x,0)\right|^2 d x=\int_{\Omega} \varphi^2 d x .
    \]At this point the claim follows by \cite[Lemma 2.4]{lin1988large} (see also \cite[Proof of Theorem 4.3]{stinga2015fractional}).

    \emph{Step 4: Conclusion:} If we define \[\Gamma=\{\gamma \in C([0,1] ; \mathrm{H}^{\varepsilon}(\mathcal{C},y^a)): \gamma(0)=0, \gamma(1)=t_0\Phi\},\]the Mountain--Pass Theorem implies the existence of a function $U^{\varepsilon,a}$ such that \[
    J'_{\varepsilon,a}(U^{\varepsilon,a})=0\quad\text{and}\quad{0<m:= J_{\varepsilon,a}(U_{\varepsilon,a})=\min _{\gamma \in \Gamma} \max _{t \in[0,1]} J_{\varepsilon,a}(\gamma(t))}.
    \]
    We immediately have
    $$J_{\varepsilon,a}(U_{\varepsilon,a})\le C\varepsilon^{\frac{n}{2}}.$$ 
    
    Let us show now that $U^{\varepsilon,a}>0$ in $\mathcal C$ and $u^{\varepsilon,a}(x)=U^{\varepsilon,a}(x,0)>0$ in $\Omega$. Using the negative part of $U^{\varepsilon,a}$ as a test function in the weak formulation (\ref{weak formulation}) we deduce that
 \[
 \begin{split}&\iint_\mathcal C (\varepsilon |\nabla_x U^{\varepsilon,a}_{-}|^2+|\partial_yU_-^{\varepsilon,a}|^2)y^a\,dx dy +\int_\Omega |U_-^{\varepsilon,a}(x,0)|^2\,dx \\
 &\hspace{2em}=\frac{1}{p+1}\int_\Omega \big(U^{\varepsilon,a}(x,0)\big)_+^{p+1}U^{\varepsilon,a}(x,0)_-\,dx=0,
 \end{split}
 \]
 which imples that $U^{\varepsilon,a} \ge 0$ in $\mathcal C$ and $u^{\varepsilon,a}\ge 0$ in $\Omega$.

Moreover, {by the strong maximum principle for degenerate elliptic equations \cite[Corollary 2.3.10]{MR643158}}, we have that $U^{\varepsilon,a}>0$ in $\mathcal C$. 

In order to prove that $u^{\varepsilon,a}>0$ in $\Omega$, we argue by contradiction. Assume that there exists  $x_0\in{\Omega}$ such that $u^{\varepsilon,a}(x_0)=0.$ Then, we can choose $r>0$ such that $B(x_0,r)\subset \Omega$, hence $U^{\varepsilon,a}$ solves \[
\begin{cases}\varepsilon^{} \Delta_x U^{\varepsilon,a}+\frac{a}{y} U^{\varepsilon,a}_y+U^{\varepsilon,a}_{y y}=0, & \text { in } \widetilde B^+({(x_0,0),r}) \\
	 U^{\varepsilon,a} > 0 &\text{ in } \widetilde B^+({(x_0,0),r})\\ -\lim _{y \rightarrow 0} y^a U^{\varepsilon,a}_y(x, y)+u^{\varepsilon,a}(x)=u^{\varepsilon,a}(x)_{+}^p & \text { in } B(x_0,r).\end{cases}
\]
By Corollary \ref{u_y}, we can apply
the Hopf principle of Cabrè and Sirè \cite[Proposition 4.11]{cabre2014nonlinear} to deduce that \[
0>\lim_{y\to 0}-y^aU_y^{\varepsilon,a}(x_0,y)=(u^{\varepsilon,a}(x_0))^p-u^{\varepsilon,a}(x_0)=0,
\]which is a contradiction. 


    It remains only to show that $U_{\varepsilon,a}$ is not constant. Assume by contradiction that $U_{\varepsilon,a}=c_{\varepsilon,a}$ then we have \begin{equation}\label{first}
    J_{\varepsilon,a}(c_{\varepsilon,a})=|\Omega|\left(\frac{c_{\varepsilon,a}^2}{2}-\frac{(c_{\varepsilon,a})_+^{p+1}}{p+1}\right)
    \end{equation}and \begin{equation}\label{second}
    0=\langle J_{\varepsilon,a}'(c_{\varepsilon,a});c_{\varepsilon,a}\rangle=|\Omega|\left(c_{\varepsilon,a}^2-(c_{\varepsilon,a})_+^{p+1}\right)
    .\end{equation}
Since $p>1$ and $c_{\varepsilon,a}>0$, we deduce from (\ref{second}) that $c_{\varepsilon,a}=1$ and by (\ref{first}) we have \[
J_{\varepsilon,a}(1)=|\Omega|\left(\frac{1}{2}-\frac{1}{p+1}\right)\le C\varepsilon^{\frac{n}{2}},
\]that is not possible if $\varepsilon$ is small enough.

\end{proof}

\section{Non--existence of non--constant solutions for large \texorpdfstring{$\varepsilon$}{ε}}
{In this Section, we establish the nonexistence of non--constant solutions for $\varepsilon$ large enough. To do this, a crucial ingredient will be a uniform $L^\infty$ estimate, which is given in Theorem \ref{uniformly bounded} below.
Such an estimate will be obtained, following the approach by Gidas and Spruck \cite{GS}, by the combination of a blow-up argument and the following Liouville-type result, which can be found in \cite{JLX} (see also \cite{MR3671257}).
}

\begin{thm}\cite[Remark 1.9]{JLX}\label{liouville}
    Let $1<p<\frac{n+2s}{n-2s}$ and let  $U\in H^1_{\rm{loc}}(\mathbb R^{n+1}_+,y^a)$ be a weak solution of \[\begin{cases}\Delta_xU(x,y)-\frac{a}{y}U_y(x,y)+U_{yy}(x,y)=0 & x \in \mathbb{R}^n, y>0, \\U(x,y)\ge0&x\in\mathbb{R}^n,y>0,\\- \lim _{y \rightarrow 0+} y^a U_y(x, y)=U(x, 0)^p, & x \in \mathbb{R}^n.\end{cases}\]
    
    Then, $U\equiv 0.$
\end{thm}

The following Caccioppoli--type inequality will be useful later on. {The proof can be found in \cite[ Lemma 3.2.]{caffarelli2016fractional}.

 Let $\mathbf B(x)$ be a symmetric, uniformly elliptic and uniformly bounded $n\times n$ matrix, and let $\mathbf A(x)$ be of the form

\begin{equation}\mathbf A(x)=	\begin{bmatrix}
	\mathbf B(x) & 0 \\
	0 & 1
\end{bmatrix}\label{matrix}.
\end{equation}
}

\begin{lem}{\cite[ Lemma 3.2.]{caffarelli2016fractional}}\label{cacioppoli} Let $U$ be a solution of \begin{equation}\label{cacioppoli-system}
\begin{cases}
\operatorname{div}\left(y^a \mathbf{A}(x)\nabla U\right)=0 &\text { in } \Omega\times(0,+\infty) \\
- \lim _{y \rightarrow 0+}y^a U_y(x ,0)=f(x) &\text { in } \Omega,
\end{cases}
\end{equation}where the matrix $\mathbf{A}$ is as in \eqref{matrix}. Then for every $R>0$ such that $\widetilde B^+_{2R}\subseteq\Omega\times(0+\infty)$ we have \[
\iint_{\widetilde B_R{ }^{+}} |\nabla U|^2y^adxdy \leq \frac{C}{R^2} \int_{\widetilde B_{2 R}^{+}} U^2y^adxdy+C\int_{B_{2 R}} |f(x) U(x,0)|dx.
\]
\end{lem}

%

We can now give the $L^\infty$ a--priori estimate for positive solutions to \eqref{main system}, which will play a crucial role in our nonexistence result (Theorem \ref{non-existence} below).

\begin{thm}\label{uniformly bounded}
    There exists a positive constant $C>0$, that does not depend on $\varepsilon$, such that, if $u$ is a weak solution of \eqref{main system},
	then,
      \[\sup_{x\in\Omega}u(x)\le C.\] 
\end{thm}
\begin{proof}
    \emph{Step 1:} First, we prove the statement for any $\varepsilon\le \varepsilon_0$, where $\varepsilon_0$ is fixed. Assume by contradiction that there exists a sequence of functions $u_k$ that solve problem (\ref{main system}) with parameters $0\le\varepsilon_k\le\varepsilon_0$ and a sequence of points $P_k\in\overline\Omega$ such that $$
M_k:=\sup _{\Omega} u_k=u_k\left(P_k\right) \rightarrow \infty, \quad \text { and } \quad P_k \rightarrow P \in \overline{\Omega}, \quad \text { as } k \rightarrow \infty.
$$
{Let $U_k$ be the extension of  $u_k$ given in Theorem \ref{thm-extension}: we have that $U_k(x,0)=u_k(x)$ and $U_k$ satisfies (\ref{extended non linear})}. 

{By the strong maximum principle for degenerate elliptic equations \cite[Corollary 2.3.10]{MR643158} and Hopf's boundary Lemma (applied to points on the lateral boundary of the cylinder $\partial_L \mathcal C$)}, we deduce that the maximum of $U_k$ is attained on $\overline{\Omega}\times\{0\};$ therefore \[
\sup _{\overline{\mathcal{C}}} U_k=U_k\left(P_k, 0\right)=M_k.
\]We distinguish two cases, depending on where $P$ is located.

\emph{Case 1}: $P\in\partial\Omega.$ Since $\partial\Omega$ is regular, we straighten the boundary of $\Omega$ near $P$ with a local diffeomorphism $\Psi$ such that $\Psi (P)=0$. If we call $z=\Psi(x)$ and we define $\widetilde{U}_k(z,y)=U_k(\Psi^{-1}(z),y)$ we find that $\widetilde U_k$ solves the system \[\begin{cases}\varepsilon^{} \Div_z(\mathbf{B}(z)\nabla_z\widetilde U_k) +\frac{a}{y}(\widetilde U_k)_y+(\widetilde U_k)_{y y}=0, & \text { in } B_{2\delta}\cap \{z_n>0\}\times (0,+\infty), \\ \partial_\nu \widetilde U_k=0, & \text { on }  B_{2\delta}\cap \{z_n=0\}\times [0,+\infty), \\ -\lim_{y\to 0}y^a(\widetilde U_k)_y(z, y)+\widetilde U_k(z, 0)=\widetilde U_k(z, 0)_+^p, & \text { in }  B_{2\delta}\cap \{z_n>0\},\end{cases}\]where $\delta>0$ is a small parameter and $\mathbf{B}$ is the matrix \[
\mathbf{B}(z)=\frac{J\Psi(\Psi^{-1}(z))J\Psi(\Psi^{-1}(z))^t}{|\det J\Psi(\Psi^{-1}(z))|}
\]that is uniformly elliptic and uniformly bounded. Let $(q'_k,\alpha_k)=Q_k=\Psi(P_k)$ with $\alpha_k\ge 0$. Since $Q_k\to 0$ we can assume $|Q_k|<\delta$. Now we define \[
{\lambda_k:=\frac{\varepsilon_k^\frac{1}{2}}{M_k^{\frac{p-1}{2s}}}},
\]
and since $\varepsilon_k$ is bounded and $M_k\to\infty$ we have that $\lambda_k$ goes to 0. Now we distinguish two cases.

\emph{Sub--case 1.1} $\frac{\alpha_k}{\lambda_k}$ remains bounded as $k\to+\infty$. We can assume that $\frac{\alpha_k}{\lambda_k}\to\alpha\ge 0.$ If we define the rescaled function as \[
{W_k(z,y):=\frac{1}{M_k}\widetilde{U}_k\left(\lambda_kz'+q_k',\lambda_kz_n,\frac{y}{{M_k^{\frac{p-1}{2s}}}}\right)},\quad z=(z',z_n)\in B_{\frac{\delta}{\lambda_k}}\cap\{z_n>0\},\:y>0,
\]we can see that $0<W_k \le 1$ for any $k$ and satisfies 
\begin{equation}\label{rescaled}
\begin{cases} \Div_z(\widetilde {\mathbf{B}}(z)\nabla_z W_k) +\frac{a}{y}( W_k)_y+( W_k)_{y y}=0, & \text { in } B_{\frac{\delta}{\lambda_k}}\cap \{z_n>0\}\times (0,+\infty), \\ \partial_\nu  W_k=0, & \text { on }  B_{\frac{\delta}{\lambda_k}}\cap \{z_n=0\}\times [0,+\infty), \\ -\lim_{y\to 0}y^a( W_k)_y(z, y)+M_k^{-p+1} W_k(z, 0)= W_k(z, 0)_+^p, & \text { in }  B_{\frac{\delta}{\lambda_k}}\cap \{z_n>0\},\end{cases}
\end{equation}where $\widetilde {\mathbf{B}}(z)=\mathbf{B}(\lambda_kz'+q_k',\lambda_kz_n).$


%

Now we want to investigate the limit of $W_k$ as $k$ goes to infinity. In order to do this, we set \[
f_k(x)=W_k(z,0)_+^p-M_k^{-p+1}W_k(z,0).
\]
	

 Since $\frac{\delta}{\lambda_k}\to\infty$ we can choose $R$ such that $\widetilde{B}_{4R}^+\subseteq \widetilde{B}_{\frac{\delta}{\lambda_k}}^+$. By $(ii)$ of Theorem \ref{first regularity} and $(i)$ of Theorem \ref{second regularity} we have that $f_k$ are equi--bounded in $C^{0,\alpha}(\overline {B_{4R}}).$ 

We can use Lemma \ref{cacioppoli} to infer that $\nabla W_k$ are equi--bounded in $L^2(\overline {\widetilde{B}_R^+},y^a)$. Therefore, up to subsequence, the functions $ W_k$ converges to a function $W$ weakly in $H^1(y^a,\widetilde{B}_R^+)$. Now let $R_j$ be a sequence such that  $R_j\to\infty$ {(and $\widetilde{B}_{4R_j}^+\subseteq \widetilde{B}_{\frac{\delta}{\lambda_k}}^+$)}, arguing as before we can find, for every $j$, a subsequence of $ W_k$ such that $W_{k_i(j)}$ weakly converges to a function $W_j$ in $H^1(y^a,\overline{\widetilde{B}^+_{R_j}}).$ By a diagonal argument we can find a subsequence $W_{k_i}$ that weakly converges to $W$ in $H^1_{\text{loc}}(y^a,\mathbb R ^{n+1}_+).$ Passing to the limit in the weak formulation of (\ref{rescaled}) it is easy to see that the limit function $W$ is a non-negative solution to \begin{equation}\label{soluzione nulla}
\begin{cases} \Delta_x W +\frac{a}{y} W_y+ W_{y y}=0, & \text { in } \mathbb{R}_+^n\times (0,+\infty), \\
	\partial_\nu W=0 & \text{ on } \partial \mathbb R^n_+\times (0+\infty)\\
	 -\displaystyle{\lim_{y\to 0}y^a W_y(z, y)= W(z, 0)^p} & \text { in }  \mathbb{R}^n_+.\end{cases}
\end{equation}

Let now consider $W^*$ to be the even reflection of $W$ through $\{z_n=0\}$, then $W^*$ satisfies

\begin{equation}
	\begin{cases} \Delta_x W^* +\frac{a}{y} W^*_y+ W^*_{y y}=0, & \text { in } \mathbb{R}^{n+1}_+, \\
	-\displaystyle{\lim_{y\to 0}y^a W^*_y(z, y)= W(z, 0)^p} & \text { in }  \mathbb{R}^n.
\end{cases}
\end{equation}

By Theorem \ref{liouville}, we get that $W^*$ is identically $0$, but this is not possible since \[
W^*(0',\alpha,0)=\lim_{k\to\infty}W_k\left(0',\frac{\alpha_k}{\lambda_k},0\right)=\lim_{k\to\infty}\frac{1}{M_k}\widetilde U_k(Q_k,0)=1
.\]\emph{Sub--case 1.2} $\frac{\alpha_k}{\lambda_k}$ is unbounded. We can assume that $\frac{\alpha_k}{\lambda_k}\to\infty.$ If we define \[
W_k(z,y)=\frac{1}{M_k}\widetilde U_k\left(\lambda_kz+Q_k, \frac{y}{M_k^{{\frac{p-1}{2s}}}}\right)
\]we find that $W_k$ solves a problem that is similar to (\ref{rescaled}) but with Neumann boundary conditions on $\{z_n=-\frac{\alpha_k}{\lambda_k}\}.$ With the same compactness argument, we find that $W_k$ converges to a solution $W$ of  (\ref{soluzione nulla}), and again, after considering the even reflection $W^*$, we deduce that $W^*=0$ which gives a contradiction as before.

\emph{Case 2:} $P\in\Omega.$ In this case we do not need to straighten the boundary. If we perform the scaling \[
W_k(x,y)=\frac{1}{M_k}U_k\left(\lambda_k x+P_k,\frac{y}{M_k^{\frac{p-1}{2s}}}\right)
\]the argument follows that of the sub--case 1.2.

\emph{Step 2:} If $\varepsilon\ge\varepsilon _0$ we can apply \emph{Step 1} to the function $u\varepsilon^{-\frac{s}{(p-1)}}$ to find that there exists a constant $C$ (independent  on $\varepsilon$) such that \begin{equation}\label{stima esponenziale}
\sup_{\Omega}u\le C\varepsilon^{\frac{s}{(p-1)}}.
\end{equation}
In what follows, $C$ will denote possibly different positive constants, which does not depend on $\varepsilon$.

\medskip
Let $U$ be th extension of $u$ satisfying \eqref{extension} and
let $\beta \ge 1$ to be chosen later. We choose $\Phi= U^{2\beta-1}$ as a test function in the weak formulation (\ref{weak formulation}) to get \begin{equation}
\begin{split}
    \label{first estimate}
& \varepsilon \frac{2 \beta-1}{\beta^2} \iint_{\mathcal{C}}|\nabla_x(U^\beta)|^2 y^ad x d y+\frac{2 \beta-1}{\beta^2} \iint_{\mathcal{C}}|\partial_yU^\beta|^2 y^ad x d y+\int_{\Omega}|U^\beta(x, 0)|^2 d x \\
&\hspace{2em}
 =\int_{\Omega}\left|U(x, 0)\right|^{p-1+2\beta} d x .
 \end{split}
\end{equation}
Thanks to (\ref{stima esponenziale}) we can estimate the right hand side of (\ref{first estimate}) as \begin{equation}\label{intermediate estimate}
\int_{\Omega}\left|U(x, 0)\right|^{p-1+2\beta} d x \le C^{p-1}\varepsilon^s\int_{\Omega}u^{2\beta}dx.
\end{equation}If we use Proposition \ref{norma} and equation (\ref{equivalence}), we get 
\begin{equation}
    \label{second estimate}
C_s\iint_{\mathcal{C}}\left(\varepsilon^{}\left|\nabla_x U^\beta\right|^2+\left|\partial_y U^\beta\right|^2\right) y^a d x d y\ge
\left\|(-\varepsilon \Delta)^{\frac{s}{2}} u^\beta\right\|_{L^2(\Omega)}^2 = \varepsilon^{s}[u^\beta]^2_{H^s(\Omega)} .
\end{equation}From equations (\ref{first estimate}), (\ref{intermediate estimate}) and (\ref{second estimate}), and the fact that $\frac{\beta^2}{2\beta-1}\le\beta$ for $\beta\ge 1$, {we get 

\[
\|u^\beta\|^2_{H^s(\Omega)}\le C\beta\int_{\Omega}u^{2\beta}.
\]
}
By the fractional Sobolev embedding $H^s(\Omega)\hookrightarrow L^{2^*_s}(\Omega)$ (where $2^*_s=2n/(n-2s)$), we get \[
\left(\int_{\Omega} u^{\beta 2^*_s} d x\right)^{\frac{2} {2^*_s}} \leq C \beta \int_{\Omega} u^{2 \beta} d x,
\]and arguing as in \cite[Pages 21--22]{lin1988large} (see also \cite[Pages 1038--1039]{stinga2015fractional}), we finally deduce that \[
\|u\|_{L^{\infty}(\Omega)} \leq C|\Omega|^{1 / p},
\]which concludes the proof.\end{proof}

We can now prove our non-existence result for $\varepsilon$ sufficiently large.
\begin{thm}\label{non-existence}
    There exists $\varepsilon^*>0$ such that if $\varepsilon>\varepsilon^*$ then the unique solution of problem (\ref{main system}) is $u\equiv1.$
\end{thm}
\begin{proof}
    Assume by contradiction that $u$ is a non constant solution  of (\ref{main system}).  Writing $u=\phi+u_\Omega$, where $u_\Omega:=\frac{1}{|\Omega|}\int_\Omega u$, we can see that $\phi$ solves the following equation \[
    \left(-\varepsilon \Delta\right)^{s} \phi+\phi-\left(\int_0^1 p\left(u_{\Omega}+t \phi\right)^{p-1} d t\right) \phi=u_{\Omega}^p-u_{\Omega} .
    \]
    Let $\Phi$ be the extension of $\phi$ (given in Theorem \ref{thm-extension}), which satisfies  
    \begin{equation}\label{extension zero average}
    \begin{cases}\varepsilon^{} \Delta_x \Phi+\frac{a}{y} \Phi_y+\Phi_{y y}=0, & \text { in } \mathcal{C}, \\ \partial_\nu \Phi=0, & \text { on } \partial_L \mathcal{C}, \\ -\lim _{y \rightarrow 0}y^a \Phi_y=\left(\int_0^1 p\left(u_{\Omega}+t \phi\right)^{p-1} d t\right) \phi-\phi+u_{\Omega}^p-u_{\Omega}&\text{ in } \Omega .\end{cases}
    \end{equation}
Testing the weak formulation for problem (\ref{extension zero average}) with $\Phi$ itself, and taking into account that $\phi$ has zero average we find:\[
    \varepsilon \iint_{\mathcal{C}}\left|\nabla_x \Phi\right|^2y^a d x d y+\iint_{\mathcal{C}}\left|\Phi_y\right|^2y^a d x d y+\int_{\Omega} \phi^2 d x=\int_{\Omega}\left(\int_0^1 p\left(u_{\Omega}+t \phi\right)^{p-1} d t\right) \phi^2 d x .
    \]By the uniform $L^\infty$ estimate of Theorem \ref{uniformly bounded}, we deduce \[
     \varepsilon \iint_{\mathcal{C}}\left|\nabla_x \Phi\right|^2y^a d x d y+\iint_{\mathcal{C}}\left|\Phi_y\right|^2y^a d x d y+\int_{\Omega} \phi^2 d x\le p C^{p-1}\int_{\Omega} \phi^2 d x ,
    \]and by Proposition \ref{norma} and equation (\ref{equivalence}) we get \[
    \frac{\varepsilon^s}{C_s} [\phi]^2_{H^s(\Omega)}+\int_\Omega\phi^2dx\le pC^{p-1}\int_\Omega\phi^2dx.
    \]
    The Poincarè inequality for fractional Sobolev spaces leads to \[
    \varepsilon^s\int_\Omega\phi^2dx\le \widetilde C_s(pC^{p-1}-1)\int_\Omega\phi^2dx
    ,\]that is not possible if $\varepsilon$ is large enough.
\end{proof}

\section{Shape of solutions} 
{In this Section we prove that, if $\varepsilon$ is sufficiently small, any solution $u^{\varepsilon,a}$ of \eqref{main system} given by Theorem \ref{existence result} concentrates around some points and tends to zero in measure as $\varepsilon \to 0$. This is the content of Theorem \ref{upper level} below. A crucial ingredient in its proof is the following Harnack inequality.
}

\begin{prp}[Harnack inequality]\label{harnack}
	Let $u$ be a non--negative solution of \begin{equation}\label{harnack equation}\begin{cases}
			(-\varepsilon\Delta_N)^su+c(x)u=0&\text{in $\Omega$}\\
			\partial_\nu u=0&\text{on $\partial\Omega$},
	\end{cases}\end{equation}where $c\in C^0(\overline{\Omega})$. Then, there exists a constant $C_0=C_0\left(n,s,\Omega,\frac{\|c\|_\infty R^{2s}}{\varepsilon^s}\right)$ such that for every $x_0\in\overline{\Omega}$ and any $R>0$ we have 
	\begin{equation}\label{harnack-ineq}
		\sup_{B(x_0,R)\cap\Omega }u\leq C_0\inf_{B(x_0,R)\cap\Omega}u.
	\end{equation}
\end{prp}\begin{proof}
	Let $U$ be the extension of $u$, which satisfies \begin{equation}\label{harnack extended}
		\begin{cases} \Delta_x U+\frac{a}{y} U_y+U_{y y}=0 & \text { in } \mathcal{C} \\ \partial_\nu U=0, & \text { on } \partial_L \mathcal{C} \\ -\lim _{y \rightarrow 0} y^a U_y(x, y)+\varepsilon^{-s}c(x)U(x, 0)=0 & \text { in } \Omega. \end{cases}
	\end{equation}
Clearly it is enough to prove the existence of a constant $C_0=C_0\left(n,s,\Omega,\frac{\|c\|_\infty R^{2s}}{\varepsilon^s}\right)$ such that, for every $x_0\in \overline{\Omega}$ \[\sup_{{\widetilde B^+((x_0,0),R)\cap\mathcal{C}}}U\leq C\inf_{{\widetilde B^+((x_0,0),R)\cap\mathcal{C}}}U.
	\]
	
	
	If $x_0\in\Omega$ and $R$ is sufficiently small, then (\ref{harnack-ineq}) follows from the Harnack inequality of Cabrè and Sire \cite[Lemma 4.9]{cabre2014nonlinear}. Instead, if $x_0\in\partial\Omega$ we can follow the argument in \cite[Proof of Lemma 4.3]{lin1988large}: we straighten the boundary near $x_0$ with a local diffeomorphism $\Psi$ such that $\Psi (P)=0$. If we call $z=\Psi(x)$ and we define $\widetilde{U}(z,y)=U(\Psi^{-1}(z),y)$ we find that $\widetilde U$ solves 
     \[
	\begin{cases} \operatorname{div}_z\left(\mathbf{B}(z) \nabla_z \widetilde{U}\right)+\frac{a}{y}\widetilde{U}_y+\widetilde{U}_{y y}=0, & \text { in } (B_{\delta}\cap\left\{z_n>0\right\})\times \mathbb R^+ , \\ \partial_\nu \widetilde{U}_k=0, & \text { on } (B_{ \delta} \cap\left\{z_n=0\right\})\times \mathbb R^+ , \\\displaystyle {-\lim _{y \rightarrow 0} y^a\widetilde{U}_y(z, y)+\varepsilon^{-s}\widetilde d(z)\widetilde{U}(z, 0)=0,} & \text { in } B_{ \delta} \cap\left\{z_n>0\right\}.\end{cases}
	\]
	where $d(z)=c(\Psi^{-1}(z))$. Since $\partial_\nu\widetilde U=0$ on $B_\delta\cap\{z_n=0\}$, we can extend $\widetilde U$ (by even reflection) to a function $U^*$ that is a solution of 
	\begin{equation}\label{eq-coefficients}
		\begin{cases} \operatorname{div}_z\left(\widetilde {\mathbf{B}}(z) \nabla_z U^*\right)+\frac{a}{y}{U^*}_y+U^*_{y y}=0, & \text { in } \widetilde B_{\delta}^+ ,\\ \displaystyle{-\lim _{y \rightarrow 0} y^a{U^*}_y(z, y)+\varepsilon^{-s} d^*(z){U^*}(z, 0)=0,} & \text { in } B_{ \delta} ,\end{cases}
	\end{equation}
	where $ d^*$ is the even reflection of $\widetilde d$ through $\{x_n=0\}$. Hence, one can argue as in \cite[Proof of Lemma 4.3]{lin1988large} and use the Harnack inequality of Lemma \ref{harnack-coefficients} below for the extended problem with bounded measurable coefficients (just depending on the horizontal variable), to conclude.
\end{proof}

	\begin{lem}\label{harnack-coefficients}
		Let $d \in L^\infty (B_{4R})$ and let $\widetilde{\mathbf{B}}(z)$ be a uniformly elliptic and uniformly bounded symmetric matrix with bounded measurable coefficients. Let $\varphi \in H^1(\widetilde B^+_{4R}, y^a)$ be a non--negative weak solution to 
		
		\begin{equation}
			\begin{cases}
				\operatorname{div}_z\left(\widetilde {\mathbf{B}}(z) \nabla_z U\right)+\frac{a}{y}{U}_y+U_{y y}=0, & \text { in } \widetilde B^+_{4R}  ,\\ \displaystyle{-\lim _{y \rightarrow 0} y^a{U}_y(z, y)+d(z){U}(z, 0)=0,} & \text { in } B_{4R}.\end{cases}
		\end{equation}

		Then,
		
		$$\sup_{\widetilde B_{R}^+} U\le C\inf_{\widetilde B_R^+}U,$$
		where $C$ is a constant depending only on $n,\,s,\,R^{2s}\|d\|_{L^\infty(B_{4R})}$ and the ellipticity of the matrix $B$.
		
	\end{lem}

 Lemma \ref{harnack-coefficients} is the analog of \cite[Lemma 4.9]{cabre2014nonlinear}  for weak solutions of \eqref{eq-coefficients}, with the difference that here we have some coefficients $\widetilde{\mathbf{B}}(z)$. We omit the proof since, as one can easily verify, it is exactly the same as the one in \cite{cabre2014nonlinear}, being $\widetilde{\mathbf{B}}(z)$  bounded elliptic and depending only on the horizontal variable $z$.
    
	The following Lemma provides  $L^2$ and $L^p$ estimates (in terms of $\varepsilon$) for the solutions to \eqref{extended non linear} given by Theorem \ref{existence result}, which will be useful in the sequel. 

\begin{lem}
	\label{shape_1} Let $U^{\varepsilon,a}\in \mathrm H^\varepsilon(\mathcal{C},y^a)$ be any solution of problem \eqref{extended non linear} given by Theorem \ref{existence result} and let $u^{\varepsilon,a}$ be its trace over $\Omega$.
	
	Then, there exists a constant $C>0$ independent on $\varepsilon$ such that \[
	\varepsilon^{} \iint_\mathcal C |\nabla_x U^{\varepsilon,a}|^2y^adxdy+\iint_\mathcal C |U^{\varepsilon,a}_y|^2y^adxdy +\int_\Omega |u^{\varepsilon,a}|^2dx=\int_\Omega |u^{\varepsilon,a}|^{p+1}dx \leq C \varepsilon^{\frac{n}{2} }.
	\]
\end{lem}
\begin{proof}
	If we use $U^{\varepsilon,a}$ as a test function in the weak formulation of problem \eqref{extended non linear}, we find \[
	\varepsilon^{} \iint_\mathcal C |\nabla_x U^{\varepsilon,a}|^2y^adxdy+\iint_\mathcal C  |U^{\varepsilon,a}_y|^2y^adxdy +\int_\Omega |u^{\varepsilon,a}|^2dx=\int_\Omega |u^{\varepsilon,a}|^{p+1}dx.
	\]
	Since $U^{\varepsilon,a}$ is a solution given by Theorem \ref{existence result}, it satisfies the energy estimate \[
	\frac{\varepsilon^{}}{2} \iint_\mathcal C |\nabla_x U^{\varepsilon,a}|^2y^adxdy+\frac{1}{2}\iint_\mathcal C |U^{\varepsilon,a}_y|^2y^adxdy +\frac{1}{2}\int_\Omega |u^{\varepsilon,a}|^2dx-\frac{1}{p+1}\int_\Omega |u^{\varepsilon,a}|^{p+1}dx\le C\varepsilon^{\frac{n}{2}}.
	\]
	Hence, we obtain \[
	\int_\Omega |u^{\varepsilon,a}|^{p+1}dx\le C\varepsilon^{\frac{n}{2}}+\frac{2}{p+1}\int_\Omega |u^{\varepsilon,a}|^{p+1}dx
	,\]and the conclusion follows observing that $\frac{2}{p+1}<1$ for any  $p>1$.
\end{proof}

By using an interpolation argument, we have the following:

\begin{lem}\label{shape_2}
	Let $u^{\varepsilon,a}$ and $U^{\varepsilon,a}$ be as in the previous Lemma. Then,  for every $q\in[2,\frac{2n}{n-2s}]$ there exits a constant $C_q$ independent on $\varepsilon$ such that \[
	\int _\Omega |u^{\varepsilon,a}|^qdx\leq C_q\varepsilon^{\frac {n} {2}}.
	\]
\end{lem}\begin{proof}
	By Lemma \ref{shape_1}, the statement is true for $q=2$; hence, by a standard interpolation argument, we only need to show the estimate for $q=\frac{2n}{n-2s}$. Let then $q=\frac{2n}{n-2s}$, using the fractional Sobolev inequality \eqref{trace-Sobolev} and Theorem \ref{existence result}, we have \[
	\int_\Omega |u^{\varepsilon,a}|^qdx\leq\left(C\varepsilon^{-\frac{s}{2}}\|U\|_{\varepsilon,a}\right)^q\leq C\varepsilon^{-\frac{sq}{2}}\varepsilon^{\frac{nq}{4}}=C\varepsilon^{q(\frac{n}{4}-\frac{s}{2})}
	.\]%
	Finally, being $q=\frac{2n}{n-2s}$, we have \[
	q\left(\frac{n}{4}-\frac{s}{2}\right)=\frac{n}{2},
	\]hence the conclusion follows.
\end{proof}
The following Lemma provides an $L^1$ estimate, which will be crucial in the proof of Theorem \ref{upper level} below.

\begin{lem}\label{shape_3}
	Let $u^{\varepsilon,a}$ be as in the previous Lemma. Then, there exists a constant $C_1$ independent on $\varepsilon$ such that \[
	\int_\Omega u^{\varepsilon,a}dx\le C_1\varepsilon^{\frac{n}{2}}.
	\]
\end{lem}\begin{proof}
	If we use the constant function 1 (that belongs to $\mathrm{H}^{\varepsilon}(\mathcal{C},y^a ))$ as a test function in the weak formulation (\ref{weak formulation}) we find that
	\[
	\int_\Omega u^{\varepsilon,a}dx=\int_\Omega (u^{\varepsilon,a})^pdx.
	\]
	
	Let $t=\frac{1}{p}$ and write $p=t+(1-t)(p+1)$.
	
	Using H\"older inequality  and Lemma \ref{shape_2}, we obtain 
	\[\begin{split}\int_{\Omega} (u^{\varepsilon,a})^p d x &\leq\left(\int_{\Omega} u^{\varepsilon,a} dx\right)^t\left(\int_{\Omega} (u^{\varepsilon,a})^{p+1} dx \right)^{1-t}\\
			& \leq\left( \int_{\Omega} u^{\varepsilon,a}dx\right)^t\left(C \varepsilon^{\frac{n}{2}}\right)^{1-t}=\left( \int_{\Omega} (u^{\varepsilon,a})^pdx\right)^t\left(C \varepsilon^{\frac{n}{2}}\right)^{1-t},\end{split}
	\]
that implies
\[
	\int_\Omega u^{\varepsilon,a}dx=\int_\Omega (u^{\varepsilon,a})^pdx\leq C_1\varepsilon^{\frac{n}{2}}.
	\]
\end{proof}

In order to state our last result, let us introduce some notation.
For $K=\left(k_1 \ldots, k_n\right) \in \mathbb{Z}^n$ and $l>0$, define the cube of $\mathbb{R}^n$
	$$
	Q_{K, l}=\left\{\left(x_1, \ldots, x_n\right) \in \mathbb{R}^n:\left|x_i-l k_i\right| \leq \frac{l}{2}, 1 \leq i \leq n\right\}
	.$$
    Moreover, if $u^{\varepsilon,a}$ is a solution of (\ref{main system}) given by Theorem \ref{existence result}, for every $\eta>0$ we define the super--level set of $u^{\varepsilon,a}$ as \[
	\Omega_\eta=\left\{x \in \Omega: u^{\varepsilon,a}(x)>\eta\right\} .
	\]

We can now give the main result of this Section.

\begin{thm}\label{upper level}
	
	For $\varepsilon$ small enough, let $u^{\varepsilon,a}\in \mathcal{H}_{\varepsilon}^s(\Omega)$ be a solution of (\ref{main system}) given by Theorem \ref{existence result}. Then, for every $\eta >0$, there exists a constant $C$ that depends on $n, \Omega, s$ and $\eta$ but not on $\varepsilon$, such that $\Omega_\eta$ can be covered by at most $C$ of the cubes $Q_{K,\sqrt \varepsilon}.$
\end{thm}

\begin{proof}First of all we have that $u^{\varepsilon,a}$ solves (\ref{main system}) if and only if it solves (\ref{harnack equation}) with $c(x)=1-(u^{\varepsilon,a}(x))^{p-1}$. Observe that, thanks to Theorem \ref{golbal holder regularity}, $c\in C^{0}(\overline{\Omega})$ and it is uniformly bounded in $\varepsilon$ by Theorem \ref{uniformly bounded}.
	If we choose $R=\sqrt{n\varepsilon}$ we have $\varepsilon^{-s}R^{2s}=n^s$, hence the Harnack constant of Proposition \ref{harnack} (with such a choice of the radius $R$) is independent on $\varepsilon$. If $z\in \Omega_\eta$ we can use Proposition \ref{harnack} to infer that 
	$$
	\inf _{B(z, \sqrt{n \varepsilon}) \cap \Omega} u^{\varepsilon,a} \geq C_0^{-1} \sup _{B(z, \sqrt{n \varepsilon}) \cap \Omega} u^{\varepsilon,a}> C_0^{-1}\eta .
	$$Thus the $\sqrt{n\varepsilon}$ neighborhood of $\Omega_\eta$ is contained in $\Omega_\delta$ where $\delta=C_0^{-1}\eta.$ Let $\{Q_{K^j,\sqrt{\varepsilon}}\}_{j=1}^{j=m}$ be the family of cubes whose intersection with $\Omega_\eta$ is non--empty; then we have  
	\[
	\bigcup_{j=1}^m Q_{K^j,\sqrt{\varepsilon}} \subset \Omega_\delta .
	\]
	{Such inclusion, together with the fact that, by definition, the cubes $Q_{K^j,\sqrt{\varepsilon}}$ are pairwise disjoint up to a set of measure $0$, implies that}
	\[
	m\varepsilon^{\frac{n}{2}}=\bigg|\bigcup_{j=1}^m Q_{K^j,\sqrt{\varepsilon}}\bigg|\leq |\Omega_\delta|
	.\]If we use Lemma \ref{shape_3} we have \[
	C_0^{-1}\eta|\Omega_\delta|=\delta|\Omega_\delta|=\int_{\Omega_\delta}\delta dx\leq\int_{\Omega_\delta}u^{\varepsilon,a}dx+\int_{\Omega\setminus\Omega_\delta}u^{\varepsilon,a}dx=\int_\Omega u^{\varepsilon,a}dx\leq C_1\varepsilon^{\frac{n}{2}},
	\]and combining the two above inequalities we find \[
	m\le C_0C_1 \eta^{-1},
	\]
    where $C_0$ is the constant in the Harnack inequality of Proposition \ref{harnack} and $C_1$ is the constant in the $L^1$-estimate of Lemma \ref{shape_3}.
\end{proof}

\begin{observation}
	Since $u^{\varepsilon,a}\in C^{1,\alpha}(\overline{\Omega})$ we have that $u^{\varepsilon,a}$ admits at least one maximum point. Assume that $P_\varepsilon$ and $Q_\varepsilon$ are two maximum points, if $\varepsilon$ is small enough we can use Theorem \ref{upper level} to infer that \[
	|P_\varepsilon-Q_\varepsilon|\leq \sqrt{\varepsilon}m\to0\quad\text{as $\varepsilon\to 0$}.
	\] 
\end{observation}
\begin{ack}
    The authors are members of the Gruppo Nazionale per l’Analisi Matematica, la Probabilità e le loro
   Applicazioni (GNAMPA) of the Istituto Nazionale di Alta Matematica (INdAM). The authors were partially founded by the INdAM--GNAMPA Project "Ottimizzazione Spettrale, Geometrica e Funzionale", CUP: E5324001950001, and by the PRIN project 2022R537CS "NO$^3$--Nodal Optimization, NOnlinear elliptic equations, NOnlocal geometric problems, with a focus on regularity", CUP: J53D23003850006. E.C. is supported by  grants MTM2017-84214-C2-1-P and RED2018-102650-T funded by MCIN/AEI/10.13039/501100011033 and by ``ERDF A way of making Europe'', and
is part of the Catalan research group 2017 SGR 1392.
\end{ack}

\bibliographystyle{abbrv}
\bibliography{cas-refs}

\begin{thebibliography}{10}

\bibitem{AV}
N.~Abatangelo and E.~Valdinoci.
\newblock Getting acquainted with the fractional {L}aplacian.
\newblock In {\em Contemporary research in elliptic {PDE}s and related topics}, volume~33 of {\em Springer INdAM Ser.}, pages 1--105. Springer, Cham, 2019.

\bibitem{ambrosetti1973dual}
A.~Ambrosetti and P.~H. Rabinowitz.
\newblock Dual variational methods in critical point theory and applications.
\newblock {\em Journal of functional Analysis}, 14(4):349--381, 1973.

\bibitem{BCdPS}
C.~Br\"{a}ndle, E.~Colorado, A.~de~Pablo, and U.~S\'{a}nchez.
\newblock A concave-convex elliptic problem involving the fractional {L}aplacian.
\newblock {\em Proc. Roy. Soc. Edinburgh Sect. A}, 143(1):39--71, 2013.

\bibitem{cabre2014nonlinear}
X.~Cabr{\'e} and Y.~Sire.
\newblock Nonlinear equations for fractional laplacians, i: Regularity, maximum principles, and hamiltonian estimates.
\newblock In {\em Annales de l'Institut Henri Poincar{\'e} C, Analyse non lin{\'e}aire}, volume~31, pages 23--53. Elsevier, 2014.

\bibitem{caffarelli2016fractional}
L.~A. Caffarelli and P.~R. Stinga.
\newblock Fractional elliptic equations, caccioppoli estimates and regularity.
\newblock {\em Annales de l'Institut Henri Poincar{\'e} C, Analyse non lin{\'e}aire}, 33:767--807, 2016.

\bibitem{warna}
B.~Claus and M.~Warma.
\newblock Realization of the fractional {L}aplacian with nonlocal exterior conditions via forms method.
\newblock {\em J. Evol. Equ.}, 20(4):1597--1631, 2020.

\bibitem{Davies}
E.~B. Davies.
\newblock {\em Heat kernels and spectral theory}, volume~92 of {\em Cambridge Tracts in Mathematics}.
\newblock Cambridge University Press, Cambridge, 1989.

\bibitem{DROV}
S.~Dipierro, X.~Ros-Oton, and E.~Valdinoci.
\newblock Nonlocal problems with {N}eumann boundary conditions.
\newblock {\em Rev. Mat. Iberoam.}, 33(2):377--416, 2017.

\bibitem{Escudero2006}
C.~Escudero.
\newblock The fractional {K}eller-{S}egel model.
\newblock {\em Nonlinearity}, 19(12):2909--2918, 2006.

\bibitem{MR643158}
E.~B. Fabes, C.~E. Kenig, and R.~P. Serapioni.
\newblock The local regularity of solutions of degenerate elliptic equations.
\newblock {\em Comm. Partial Differential Equations}, 7(1):77--116, 1982.

\bibitem{GS}
B.~Gidas and J.~Spruck.
\newblock Global and local behavior of positive solutions of nonlinear elliptic equations.
\newblock {\em Comm. Pure Appl. Math.}, 34(4):525--598, 1981.

\bibitem{Grubb1}
G.~Grubb.
\newblock Local and nonlocal boundary conditions for {$\mu$}-transmission and fractional elliptic pseudodifferential operators.
\newblock {\em Anal. PDE}, 7(7):1649--1682, 2014.

\bibitem{Grubb}
G.~Grubb.
\newblock Regularity of spectral fractional {D}irichlet and {N}eumann problems.
\newblock {\em Math. Nachr.}, 289(7):831--844, 2016.

\bibitem{Huaroto-Neves}
G.~Huaroto and W.~Neves.
\newblock Solvability of the fractional hyperbolic {K}eller-{S}egel system.
\newblock {\em Nonlinear Anal. Real World Appl.}, 74:Paper No. 103957, 28, 2023.

\bibitem{JLX}
T.~Jin, Y.~Li, and J.~Xiong.
\newblock On a fractional {N}irenberg problem, part {I}: blow up analysis and compactness of solutions.
\newblock {\em J. Eur. Math. Soc. (JEMS)}, 16(6):1111--1171, 2014.

\bibitem{lin1988large}
C.-S. Lin, W.-M. Ni, and I.~Takagi.
\newblock Large amplitude stationary solutions to a chemotaxis system.
\newblock {\em Journal of Differential Equations}, 72(1):1--27, 1988.

\bibitem{lions1961problemes}
J.-L. Lions and E.~Magenes.
\newblock Problèmes aux limites non homogènes et applications, vol. 1.
\newblock In {\em Travaux et Recherches Mathématiques}, volume~17, 1968.

\bibitem{Nekvinda}
A.~Nekvinda.
\newblock Characterization of traces of the weighted {S}obolev space {$W^{1,p}(\Omega,d^\epsilon_M)$} on {$M$}.
\newblock {\em Czechoslovak Math. J.}, 43(118)(4):695--711, 1993.

\bibitem{NT}
W.-M. Ni and I.~Takagi.
\newblock On the {N}eumann problem for some semilinear elliptic equations and systems of activator-inhibitor type.
\newblock {\em Trans. Amer. Math. Soc.}, 297(1):351--368, 1986.

\bibitem{stein2011functional}
E.~M. Stein and R.~Shakarchi.
\newblock {\em Functional analysis: introduction to further topics in analysis}, volume~4.
\newblock Princeton University Press, 2011.

\bibitem{stinga2010fractional}
P.~R. Stinga.
\newblock {\em Fractional powers of second order partial differential operators: extension problem and regularity theory}.
\newblock PhD thesis, Universidad Aut{\'o}noma de Madrid, 2010.

\bibitem{stinga2010extension}
P.~R. Stinga and J.~L. Torrea.
\newblock Extension problem and harnack's inequality for some fractional operators.
\newblock {\em Communications in Partial Differential Equations}, 35(11):2092--2122, 2010.

\bibitem{stinga2015fractional}
P.~R. Stinga and B.~Volzone.
\newblock Fractional semilinear neumann problems arising from a fractional keller--segel model.
\newblock {\em Calculus of Variations and Partial Differential Equations}, 54:1009--1042, 2015.

\bibitem{MR3671257}
Y.~Wan and C.-L. Xiang.
\newblock Uniqueness of positive solutions to some nonlinear {N}eumann problems.
\newblock {\em J. Math. Anal. Appl.}, 455(2):1835--1847, 2017.

\end{thebibliography}
\end{document}